\documentclass[11pt]{amsart} 

\usepackage{amsmath, amsthm, amscd, amsfonts, amssymb, enumerate, verbatim, newlfont, calc, graphicx, color}
\usepackage[bookmarksnumbered, colorlinks, plainpages]{hyperref}
\usepackage{tikz}
 \usepackage{doi}

\usepackage{amsmath, amsthm, amscd, amsfonts, amssymb, enumerate,verbatim, newlfont, calc, graphicx, color}
\usepackage[bookmarksnumbered, colorlinks, plainpages]{hyperref}
\usepackage{tikz} 
\usepackage{doi}
\usepackage{cancel}
\newtheorem{theorem}{Theorem}[section] 
 
\newtheorem{fact}[theorem]{Fact}

\newtheorem{lemma}[theorem]{Lemma}
\newtheorem{proposition}[theorem]{Proposition}

\newtheorem{corollary}[theorem]{Corollary}
\newtheorem{definition}[theorem]{Definition}

\newtheorem{example}[theorem]{Example}

\usepackage{mathtools}

\usepackage{pstricks}
\usepackage{epsfig}
\usepackage{pst-grad}
\usepackage{pst-plot}
\usepackage{subfig}

\begin{document}

\author[M. Nasernejad  and   J. Toledo]{Mehrdad  Nasernejad$^{1,2,*}$  and   Jonathan Toledo$^{3}$}
\title[Demotions of ideals in commutative rings]{Demotions of ideals in commutative rings with applications to normally torsion-freeness}


\thanks{$^*$Corresponding author}

\thanks{E-mail addresses:  m$\_$nasernejad@yahoo.com  and  jonathan.toledo@infotec.mx}  
\maketitle

\begin{center}
{\it
$^{1}$Univ. Artois, UR 2462, Laboratoire de Math\'{e}matique de  Lens (LML), \\  F-62300 Lens, France \\ 
$^{2}$Universit\'e  Caen Normandie, ENSICAEN, CNRS, Normandie Univ, GREYC UMR  6072, F-14000 Caen,  France\\
$^{3}$INFOTEC Centro de investigaci\'{o}n e innovaci\'{o}n en informaci\'{o}n \\
 y comunicaci\'{o}n, Ciudad de M\'{e}xico,14050, M\'{e}xico
}
\end{center}

\maketitle

\begin{abstract}
Let  $J \subseteq  I$ be ideals in a commutative Noetherian ring $R$, and $r,s\geq 0$. We say that  $J$ is  a  demotion of $I$ if
$I^rJ^s=I^{r+s}\cap J^s$ for all $r,s\geq 0$. In this paper, we mainly aim to explore this notion in the  polynomial rings. In particular, we investigate 
the relation between the demotion property and normally torsion-freeness. Furthermore, we compare the reductions of ideals and demotions of ideals.
\end{abstract}
\vspace{0.4cm}


\section{Introduction and Overview}

The study of powers and symbolic powers of ideals has long been a central theme in commutative algebra and its interactions with combinatorics and geometry. A classical result due to Brodmann \cite{BR} asserts that for any ideal $I$ in a Noetherian ring $R$, the sequence of associated primes $\{\mathrm{Ass}_R(R/I^s)\}_{s\geq 1}$ stabilizes for sufficiently large $s$. The limiting set, called the \emph{stable set} of associated primes, plays a fundamental role in the asymptotic behavior of ideals. Closely related to this is the notion of \emph{normally torsion-free} ideals.  An ideal $I$ is normally torsion-free if $\mathrm{Ass}_R(R/I^s)\subseteq \mathrm{Ass}_R(R/I)$ for all $s\geq 1.$ 
Although this property is uncommon in full generality, it has attracted considerable attention in the monomial setting. For example, the edge ideal (resp.\ the cover ideal) of a finite graph $G$ is normally torsion-free precisely when $G$ is bipartite \cite{GRV,SVV}. Further results extend this connection to Alexander duals of path ideals in rooted trees, strongly chordal graphs, $t$-spread principal Borel ideals, $t$-spread monomial ideals, and other families of monomial ideals, as can be found in \cite{HN, KHN1, MNQ, N3, NQ, NQKR}. More broadly, it is well-known that normally torsion-free square-free monomial ideals correspond to \emph{Mengerian} clutters, i.e., clutters satisfying the max-flow min-cut property; equivalently, for the edge ideal $I:=I(\mathcal C)$ one has $I^k=I^{(k)}$ for all $k\ge1$ if and only if $\mathcal C$ is Mengerian \cite{V1}. This perspective links commutative algebra with combinatorial optimization and polyhedral theory and places the subject in the context of the Conforti–Cornu\'e{j}ols conjecture.

Motivated by these interactions, this paper pursues three closely related directions. First, we introduce and study the notion of a demotion between ideals. 
Let  $J \subseteq  I$ be ideals in a commutative Noetherian ring $R$, and $r,s\geq 0$. Since  $I^rJ^s \subseteq I^r\cap J^s \subseteq J^s$ and $I^rJ^s\subseteq I^{r+s}$, we always have  $I^rJ^s \subseteq I^{r+s} \cap J^s$. We are interested in when  $I^rJ^s=I^{r+s}\cap J^s$ holds. 
We say that  $J$ is  a \textit{demotion} of $I$ if
$I^rJ^s=I^{r+s}\cap J^s$ for all $r,s\geq 0$. Furthermore,  $J$  is said to be   a \textit{proper demotion} of $I$ if $J$ is a demotion of $I$ and $I\neq J$. 
This equality refines the trivial containment $I^rJ^s\subseteq I^{r+s}\cap J^s$ and captures a precise compatibility between the two ideals.
In this paper, we show that in the monomial world this notion admits explicit characterizations and constructions: for instance, prime monomial ideals and several natural classes of monomial ideals give rise to demotions, and principal monomial ideals admit a simple description of all their demotions. 

The second aim is to deepen the study of normally torsion-free monomial ideals and to use the demotion framework to produce new normally torsion-free examples. Understanding the structure of normally torsion-free ideals contributes to the broader program of investigating minimal counterexamples to the Conforti–Cornu\'e{j}ols conjecture and sheds light on how combinatorial properties of hypergraphs reflect algebraic stability phenomena, more information can be found in \cite{BNT}.

The paper is organized as follows. In Section~\ref{Section 2} we collect fundamental definitions, results, and facts that
will play a key role in the subsequent sections of this paper.
 Section~\ref{Section 3} is concerned with  investigating  some general properties of demotions of ideals in commutative rings, refer to Propositions 
 \ref{Pro.General.1} and \ref{Pro.General.2}. The main aim of Section \ref{Section 4} is to identify some classes of monomial ideals that have the demotion property, consult  Propositions \ref{Demotion-1}, \ref{Demotion-4}, and \ref{Demotion-2}.  
 In Section \ref{Section 5}, we probe  the behavior of the demotions of monomial ideals under a variety of monomial operations. In particular, we explore how these ideals are affected by operations such as expansion (Proposition \ref{PRO.Expansion}), summation (Proposition \ref{Summation}), 
  weighting (Proposition \ref{PRO.Weighting}), 
 monomial multiple (Proposition \ref{Multiple}), permutation (Proposition \ref{PRO.Permutation}), monomial localization (Proposition \ref{PRO.Localization}), contraction (Proposition \ref{Cor. contraction}), and deletion (Proposition \ref{PRO.Deletion}). Building upon these operations, we propose several systematic methods for constructing new monomial ideals whose demotions are determined by those of previously studied monomial ideals. 
 Section \ref{Section 6} is associated with  exploring  the relationship between demotions of monomial ideals and normally torsion-freeness. 
 particularly, we demonstrate that the demotion property can be used to construct new normally torsion-free monomial ideals, 
   and conversely, that normally torsion-free ideals can inform the generation of ideals with the demotion property, see Theorems \ref{NTT-1} and \ref{NTT-2}. 
 Finally,   Section \ref{Section 7} is devoted to comparing  reductions and demotions of ideals such that  we establish that a number of its characteristic properties are not preserved in the case of demotions of ideals. For this purpose, we present several counterexamples in this section.

 Throughout this text, we let $\mathcal{G}(I)$ denote the unique minimal set of monomial generators of a monomial ideal $I \subset R = K[x_1, \ldots, x_n]$, where $R$ is the polynomial ring over a field $K$. Moreover, for a monomial $u \in R$, its {\em support}, denoted by $\mathrm{supp}(u)$, is defined as the set of variables that divide $u$;  we also set $\mathrm{supp}(1) = \emptyset$. For a monomial ideal $I$, we further define  
$\mathrm{supp}(I) = \bigcup_{u \in \mathcal{G}(I)} \mathrm{supp}(u).$


\section{Preliminaries} \label{Section 2}

This section aims to consolidate the fundamental definitions, key results, and essential facts and theoretical tools that will serve as the foundation for the developments discussed in the subsequent sections of this study. To establish a firm groundwork, we begin by revisiting the following result.


\begin{fact} \label{Exercise 6.4} (\cite[Exercise 6.4]{MAT})
Let $I$  and  $J$ be ideals of a Noetherian ring $A$. Prove that if $JA_\mathfrak{p} \subseteq  IA_\mathfrak{p}$  for
every $\mathfrak{p}\in \mathrm{Ass}_A(A/I)$,  then $J\subseteq I$. 
\end{fact}


\begin{fact} \label{proposition 4.3.29} (\cite[Proposition 4.3.29]{V1})
 Let $I$ be an ideal of a ring $R$. If $I$ has no embedded
primes, then $I$ is normally torsion-free if and only if $I^n = I^{(n)}$ for all $n \geq 1$. 
\end{fact}


\begin{fact} \label{exercise 6.1.25} (\cite[Exercise 6.1.25]{V1})
 If $\mathfrak{q}_1, \ldots, \mathfrak{q}_r$  are primary monomial ideals of $R$ with non-comparable
radicals and $I$ is an ideal such that $I=\mathfrak{q}_1 \cap \cdots \cap  \mathfrak{q}_r$, then
$I^{(n)}=\mathfrak{q}^n_1 \cap \cdots \cap  \mathfrak{q}^n_r$. 
\end{fact}

 
\begin{fact} \label{fact1} (\cite[Exercise 6.1.23]{V1}) 
If $I$, $J$, $L$ are monomial ideals, then  the following equalities hold:
\begin{itemize}
\item[(i)]  $I\cap (J+L)=(I \cap J) + (I \cap L).$
\item[(ii)] $I+ (J \cap L)= (I+J) \cap (I+L).$
\end{itemize}
\end{fact}


\begin{theorem} (\cite[Theorem 14.3.6]{V1})  \label{Villarreal1}
 Let $\mathcal{C}$  be a clutter and let  $A$ be its incidence matrix. The following are equivalent:
 \begin{itemize}
 \item[(i)] $\mathrm{gr}_I(R)$ is reduced, where $I=I(\mathcal{C})$ is the edge ideal of $\mathcal{C}$. 
 \item[(ii)] $R[It]$ is normal  and $\mathcal{Q}(A)$ is an integral polyhedron. 
 \item[(iii)]  $x\geq 0$; $xA\geq \textbf{1}$   is a $\mathrm{TDI}$ system. 
 \item[(iv)]  $\mathcal{C}$  has the max-flow min-cut (MFMC) property.
 \item[(v)]  $I^i = I^{(i)}$ for $i \geq 1$. 
 \item[(vi)]  $I$ is normally torsion-free, i.e.,  $\mathrm{Ass}_R(R/I^i) \subseteq  \mathrm{Ass}_R(R/I)$ for $i\geq 1$. 
 \item[(vii)]  $\mathcal{C}$  is Mengerian, i.e., $\beta_1({\mathcal{C}}^a) = \alpha_0({\mathcal{C}}^a)$ for all  $a \in \mathbb{N}^n$.
 \end{itemize}
\end{theorem}


 \begin{proposition}\label{Intersection}
  Let $I$ and $J$ be two square-free monomial ideals in a polynomial ring $R=K[x_1, \ldots, x_n]$ over a field $K$. Then, for all $k\geq 1$, we have 
  $(I\cap J)^{(k)}=I^{(k)} \cap J^{(k)}$. 
  \end{proposition}
  
  
  \begin{theorem} (\cite[Theorem 2.5]{SN})  \label{NTF.Th.2.5}
Let $I$ be  a  monomial ideal  in  $R=K[x_1, \ldots, x_n]$ such that 
$I=I_1R + I_2R$, where
 $\mathcal{G}(I_1) \subset R_1=K[x_1, \ldots, x_m]$ and $\mathcal{G}(I_2) \subset R_2=K[x_{m+1}, \ldots, x_n]$ for some  positive integer $m$. If $I_1$ and   $I_2$  are normally torsion-free, then  $I$  is so.
 \end{theorem}


\begin{lemma}\label{NCH-1}(\cite[Lemma 2.1]{NCH})
Let $I\subset R=K[x_1, \ldots, x_r]$  be a  square-free monomial ideal  and  $\mathfrak{q}$ be  a prime monomial ideal in $R$ such that $\bigcap_{\mathfrak{p}\in \mathrm{Ass}(I)}\mathfrak{p} \cap \mathfrak{q}$  is  a minimal primary decomposition of  
$I \cap \mathfrak{q}$.   Let  $I$ and $I \cap \mathfrak{q}$ be  normally torsion-free.  
Then, for all $m,n \geq 1$, we have $I^n(I\cap \mathfrak{q})^m=I^{n+m} \cap \mathfrak{q}^m$. 
  \end{lemma}
  

\begin{fact} \label{exercise 7.9.1} (\cite[Exercise 7.9.1]{MRS})
 Let $J_1, \ldots, J_n, I$  be monomial ideals of $R$, and let $f\in [[R]].$ 
\begin{itemize}
\item[(a)]  Prove that $f(\bigcap_{i=1}^nJ_i)=\bigcap_{i=1}^n(fJ_i)$.
\item[(b)] Prove or disprove: $I(\bigcap_{i=1}^nJ_i)=\bigcap_{i=1}^n(IJ_i)$. Justify your answer.
\end{itemize}
\end{fact}

\begin{lemma} \label{NTF1} (\cite[Lemma 2.17]{NQBM})
Let $I$ be a normally torsion-free square-free  monomial ideal in a polynomial ring $R=K[x_{1},\ldots ,x_{n}]$ with $\mathcal{G}(I) \subset R$. Then the ideal $$L:=IS\cap (x_{n}, x_{n+1}, x_{n+2}, \ldots, x_m)\subset  S=R[x_{n+1}, x_{n+2}, \ldots, x_m],$$  is normally torsion-free.
\end{lemma}


\section{Demotions of ideals in commutative rings} \label{Section 3}

  In this section, we investigate some general properties of demotions of ideals in commutative rings. To accomplish this, we commence by presenting the following proposition.

\begin{proposition} \label{Pro.General.1}
Let $J \subseteq I$ be ideals  in a commutative Noetherian ring $R$. Then the following statements are equivalent:
\begin{itemize}
\item[(i)]  $J$ is  a demotion of $I$;
\item[(ii)]  $W^{-1}J$ is a demotion of $W^{-1}I$ for every multiplicatively closed subset $W$ of $R$; 
\item[(iii)] $J_{\mathfrak{p}}$ is a demotion of $I_{\mathfrak{p}}$ for every prime ideal $\mathfrak{p}$ of $R$. 

\end{itemize}
\end{proposition}

\begin{proof}
(i) $\Rightarrow$ (ii). We  show that $(W^{-1}I)^r(W^{-1}J)^s=(W^{-1}I)^{r+s} \cap (W^{-1}J)^s$, where  $r,s \geq 0$ and $W$ is a  multiplicatively closed subset of $R$. Clearly,   $W^{-1}J \subseteq W^{-1}I$. As  $I^rJ^s=I^{r+s} \cap J^s$, we get the following equalities
\begin{align*}
(W^{-1}I)^r (W^{-1}J)^s=& W^{-1}(I^r) W^{-1}(J^s)\\
=&W^{-1}(I^rJ^s)\\
=&W^{-1}(I^{r+s}\cap J^s)\\
=& (W^{-1}I)^{r+s} \cap (W^{-1}J)^s.
\end{align*}
This gives  that $W^{-1}J$ is a demotion of $W^{-1}I$, as desired. 

By choosing $W=R\setminus \mathfrak{p}$, (ii) implies (iii).  To complete the proof, we need to show that (iii) implies (i). To do this, we assume that $J_{\mathfrak{p}}$ is a demotion of $I_{\mathfrak{p}}$ for  every prime ideal $\mathfrak{p}$ of $R$. Fix $r,s\geq 0$. Since $I^rJ^s \subseteq I^{r+s} \cap J^s$, it remains to prove that   $I^{r+s}\cap J^s \subseteq I^rJ^s$.  Pick an arbitrary element $\mathfrak{q}\in \mathrm{Ass}(R/I^rJ^s)$.  
Due to $J_{\mathfrak{p}}$ is a demotion of $I_{\mathfrak{p}}$ for every prime ideal $\mathfrak{p}$ of $R$, this yields that 
$I_{\mathfrak{q}}^{r+s}\cap J_{\mathfrak{q}}^s \subseteq I_{\mathfrak{q}}^rJ_{\mathfrak{q}}^s$. This leads to 
$(I^{r+s}\cap J^s)_{\mathfrak{q}} \subseteq (I^rJ^s)_{\mathfrak{q}}$. It follows at once  from Fact \ref{Exercise 6.4} that  $I^{r+s}\cap J^s \subseteq I^rJ^s$, 
as required. This finishes the proof.
\end{proof}


Recall that if $f:A \rightarrow B$  is a ring homomorphism and $I$  is an ideal of $A$, we will write $IB$ for the ideal $f(I)B$ of $B$. The ideal $IB$ is called the \textit{extension}   of $I$ to $B$, and is often denoted by $I^e$.

\begin{proposition} \label{Pro.General.2}
Let  $J \subseteq I$ be ideals  in a commutative Noetherian ring $R$, and   $f: R \rightarrow S$ be a ring homomorphism.  Denote $I^e:=IS$ and $J^e:=JS$. 
Then the following statements hold.
\begin{itemize}
\item[(i)]  If $J$ is  a demotion of $I$, then $J^k$ is  a demotion of $I^k$ for all $k\geq 1$. 
\item[(ii)] If $S$ is flat over $R$ and $J$ is a demotion of $I$ in $R$, then $J^e$ is a demotion of $I^e$ in $S$.
\item[(iii)]  If $S$ is faithfully flat over $R$ and $J^e$ is a demotion of $I^e$ in $S$, then $J$ is a demotion of $I$ in $R$. 
\end{itemize}
\end{proposition}

\begin{proof}
(i) Let  $J$ be  a demotion of $I$.  Fix $r,s\geq 0$ and $k\geq 1$. It is clear that $J^k\subseteq I^k$.  Because $J$ is  a demotion of $I$, we have $I^a J^b=I^{a+b} \cap J^b$ for all $a,b\geq 0$. 
 This implies that  $$(I^k)^r (J^k)^s=I^{kr} J^{ks}=I^{kr+ks} \cap J^{ks}= (I^k)^{r+s} \cap (J^k)^s.$$ 
This means that $J^k$ is  a demotion of $I^k$, as claimed. 

(ii)   Assume $S$ is flat and $I^rJ^s = I^{r+s}\cap J^s$ for all $r,s\geq 0$. It is obvious  that $J^e\subseteq I^e$.
Recall that the extension of ideals is given by tensoring with $S$ and then taking the image in $S$, that is, 
for any ideal $K\subseteq R$ we have $K^e=KS=\operatorname{im}(K\otimes_R S\to S)$.
Since $S$ is flat, the functor $-\otimes_R S$ is exact. Exactness gives the following standard consequences:

\begin{itemize}
\item For all ideals $X,Y\subseteq R$, $(X\cap Y)\otimes_R S \xrightarrow{\ \cong\ } (X\otimes_R S)\cap (Y\otimes_R S),$

and passing to images in $S$ yields $(X\cap Y)^e = X^e\cap Y^e.$
\item For all ideals $X,Y\subseteq R$, $(XY)\otimes_R S \xrightarrow{\ \cong\ } (X\otimes_R S)(Y\otimes_R S),$
and so $(XY)^e = X^e Y^e,$ and more generally $(X^n)^e=(X^e)^n$ for $n\geq 0$.
\end{itemize}

Apply extension to $I^rJ^s = I^{r+s}\cap J^s$, we get  $(I^rJ^s)^e = (I^{r+s}\cap J^s)^e.$
By multiplicativity and power-compatibility of extension, we have $(I^e)^r (J^e)^s = (I^rJ^s)^e.$
 Moreover, we obtain $(I^{r+s}\cap J^s)^e = (I^{r+s})^e \cap (J^s)^e = (I^e)^{r+s}\cap (J^e)^s.$
 Combining the equalities yields $(I^e)^r (J^e)^s = (I^e)^{r+s}\cap (J^e)^s$ for all $r,s\geq 0$, as required.
\medskip

(iii) Assume $S$ is faithfully flat and that  $(I^e)^r (J^e)^s = (I^e)^{r+s}\cap (J^e)^s$ for all $r,s\geq 0.$
We want to show that  $I^rJ^s = I^{r+s}\cap J^s$ for all $r,s\geq 0$. Recall this key fact for faithfully flat property that  for every ideal $K\subseteq R$, 
 one has  $(K^e)^c = K,$ where $(-)^c$ denotes contraction to $R$.  Also contraction commutes with finite intersections, that is, for any two  ideals $X,Y\subseteq S$, $(X\cap Y)^c = X^c\cap Y^c$. Apply contraction to $(I^e)^r (J^e)^s = (I^e)^{r+s}\cap (J^e)^s$, we obtain $\big((I^e)^r (J^e)^s\big)^c = \big((I^e)^{r+s}\cap (J^e)^s\big)^c.$
Using   the fact $(K^e)^c = K$ for any ideal $K\subseteq R$, we have 
$\big((I^e)^r (J^e)^s\big)^c = \big((I^rJ^s)^e\big)^c = I^rJ^s.$ Furthermore,   
\[
\big((I^e)^{r+s}\cap (J^e)^s\big)^c
= \big((I^{r+s})^e\big)^c \cap \big((J^s)^e\big)^c
= I^{r+s}\cap J^s.
\]
This leads to $I^rJ^s = I^{r+s}\cap J^s$ for all $r,s\geq 0$, which is exactly the demotion property for $J$ with respect to $I$ in $R$.
\end{proof}


\section{Some classes of monomial ideals with the demotion property} \label{Section 4}

 The main aim of this section is to identify some classes of monomial ideals that have the demotion property. To achieve this, we present the first such class in Proposition \ref{Demotion-1}. The following lemma is essential for completing the proof of Proposition \ref{Demotion-1}.

\begin{lemma} \label{lem:bounded-sum}
Let $f_1,\dots,f_m$ be nonnegative integers and let $t$ be a nonnegative integer such that $\sum_{i=1}^m f_i \ge t.$ 
Then there exist integers $c_1,\dots,c_m$ satisfying $0\le c_i\le f_i$ for all $1\leq i \leq m$ and $\sum_{i=1}^m c_i=t$.
\end{lemma}

\begin{proof}
Set $r_0:=t$ and, for each $i=1,\dots,m$, define  $c_i:=\min\{f_i,r_{i-1}\}$ and $r_i:=r_{i-1}-c_i.$ 
By construction, we have  $0\le c_i\le f_i$ and $r_i\ge0$ for all $i$. Summing the telescoping differences gives $\sum_{i=1}^m c_i = r_0 - r_m = t - r_m,$ 
so it remains to prove $r_m=0$.

Assume for contradiction that $r_m>0$. Then $r_{i-1}>0$ for every $i$, so if for some index $i$ we had $f_i>r_{i-1}$ we would get $c_i=r_{i-1}$ and hence $r_i=0$, contradicting $r_m>0$. Thus $f_i\le r_{i-1}$ for all $i$, and consequently $c_i=f_i$ for every $i$. This implies that $\sum_{i=1}^m c_i = \sum_{i=1}^m f_i \ge t,$ 
 but also $\sum_{i=1}^m c_i = t-r_m < t$, a contradiction. Hence, we must have  $r_m=0$ and $\sum_{i=1}^m c_i=t$. This finishes the proof. 
\end{proof}


\begin{proposition} \label{Demotion-1}
Let $\mathfrak{p}$ and $\mathfrak{q}$ be monomial prime ideals in 
$R = K[x_1, \dots, x_n]$ with $\mathfrak{q} \subseteq \mathfrak{p}$. 
Then $\mathfrak{q}$ is a demotion of $\mathfrak{p}$.
\end{proposition}

\begin{proof}
Fix $r,s \geq 0$. Without loss of generality, we may write
\[
\mathfrak{p} = (x_1, \dots, x_c)
\quad \text{and} \quad 
\mathfrak{q} = (x_1, \dots, x_d),
\quad \text{with } 1 \leq d \leq c \leq n.
\]
Since $\mathfrak{p}^r \mathfrak{q}^s \subseteq \mathfrak{p}^{r+s} \cap \mathfrak{q}^s$, 
it remains to show the reverse inclusion. Take a monomial $m=x_1^{\beta_1}\cdots x_n^{\beta_n}\in\mathfrak{p}^{\,r+s}\cap\mathfrak{q}^s.$
Due to  $m\in\mathfrak{q}^s$, we have $\sum_{i=1}^d\beta_i\ge s$, and since
$m\in\mathfrak{p}^{\,r+s}$, we get  $\sum_{i=1}^c\beta_i\ge r+s$. By virtue of Lemma  \ref{lem:bounded-sum}, we can choose nonnegative integers $\beta_i''\le\beta_i$ for all $1\le i\le d$ with  $\sum_{i=1}^d\beta_i''=s$.  For $n\geq i>d$ set $\beta_i''=0$, and put $\beta_i':=\beta_i-\beta_i''$ 
for all $i=1,\dots,n$. Now, define monomials
\[
m_1=x_1^{\beta_1'}\cdots x_n^{\beta_n'}\;\; \text{and} \;\;
m_2=x_1^{\beta_1''}\cdots x_d^{\beta_d''}.
\]
By construction $m=m_1m_2$, and $m_2\in\mathfrak{q}^s$ due to  $\sum_{i=1}^d\beta_i''=s$. Moreover,
\[
\sum_{i=1}^c\beta_i'=\sum_{i=1}^c\beta_i-\sum_{i=1}^c\beta_i''
=\sum_{i=1}^c\beta_i-\sum_{i=1}^d\beta_i''\ge (r+s)-s=r,
\]
since  $\beta_i''=0$ for $i>d$. Hence, we obtain $m_1\in\mathfrak{p}^r$, and so  $m=m_1m_2\in\mathfrak{p}^r\mathfrak{q}^s$. This proves that $\mathfrak{p}^{\,r+s}\cap\mathfrak{q}^s\subseteq\mathfrak{p}^r\mathfrak{q}^s$. 
Thus,  $\mathfrak{q}$ is a demotion of $\mathfrak{p}$.
\end{proof}


\begin{proposition} \label{Demotion-4}
Let $\mathfrak{p}=(x_1, \dots, x_m) \subset K[x_1, \dots, x_n]$, $1\leq m\leq n$, be a monomial prime ideal and  $J=(x_1^m, \dots, x_m^m)$.  
Then  $J$ is a demotion of $\mathfrak{p}^m$. In particular, if $m>1$, then $J$ is a proper demotion of $\mathfrak{p}^m$. 
\end{proposition}

\begin{proof}
We should prove  $(\mathfrak{p}^m)^r J^s = (\mathfrak{p}^m)^{r+s} \cap J^s$ for all $r,s\geq 0$. Fix $r,s\geq 0$. Since $(\mathfrak{p}^m)^r J^s \subseteq  (\mathfrak{p}^m)^{r+s} \cap J^s$, it suffices to check the reverse inclusion. Pick a monomial $u\in (\mathfrak{p}^m)^{r+s} \cap J^s$. 
  Write $u =(\prod_{i=1}^mx_i^{\alpha_i})(\prod_{i=m+1}^nx_i^{\alpha_i}),$ and set 
\[
\alpha = (\alpha_1,\dots,\alpha_m), \qquad
|\alpha| = \sum_{i=1}^m \alpha_i, \qquad
U_i = \Big\lfloor \tfrac{\alpha_i}{m} \Big\rfloor \; (1\leq i \leq m).
\]
In the sequel, we will use the following facts. The first fact is obvious. We  justify only the second fact:
\begin{itemize}
    \item $u \in (\mathfrak{p}^m)^t \iff |\alpha| \ge mt$. 
    \item $u \in J^t \iff \sum_{i=1}^m U_i \ge t$.
\end{itemize}

\medskip
Since the monomial  ideal $J=(x_1^m,\dots,x_m^m)$ is generated by the monomials $x_i^m$ $(1\leq i \leq m)$, a typical minimal monomial generator of $J^t$ has the form
 $g = \prod_{i=1}^m x_i^{m c_i},$ where $c_i \ge 0$ are integers with $\sum_{i=1}^m c_i = t$.  
Indeed, such $g$ arises as a product of $t$ minimal monomial generators of $J$, and conversely every product of $t$ minimal monomial generators is of this form.

\smallskip
\noindent ($\Rightarrow$) If $u \in J^t$, then some such $g$ divides $u$.  Divisibility means $\alpha_i \ge m c_i$ for all $1\leq i \leq m$, hence
 $\Big\lfloor \frac{\alpha_i}{m} \Big\rfloor \ge c_i$   for all $1\leq i \leq m$.  Summing over $i$ gives
$$\sum_{i=1}^m \Big\lfloor \frac{\alpha_i}{m} \Big\rfloor \ge \sum_{i=1}^m c_i = t.$$
\noindent ($\Leftarrow$) Conversely, suppose $\sum_{i=1}^m \Big\lfloor \tfrac{\alpha_i}{m} \Big\rfloor \ge t.$
Let $f_i := \lfloor \alpha_i / m \rfloor$ for $i=1, \ldots, m$. 
Because  $\sum_{i=1}^m f_i \ge t$,  Lemma \ref{lem:bounded-sum} permits  us to select  integers $c_1, \ldots, c_m$ such that  $0 \le c_i \le f_i$  for all $1\leq i \leq m$,  and $\sum_{i=1}^m c_i = t$. Then the monomial $g = \prod_{i=1}^m x_i^{m c_i}$ is a generator of $J^t$ and divides $u$, showing $u \in J^t$.

 This gives rise to the fact that $u \in J^t \iff  \sum_{i=1}^m \Big\lfloor \frac{\alpha_i}{m} \Big\rfloor \ge t.$

 It follows now from $u\in (\mathfrak{p}^m)^{r+s} \cap J^s$ and the above facts that  $u$ satisfies both $|\alpha| \ge m(r+s)$ and $\sum_{i=1}^m U_i \geq s.$ 
 Since $\sum_{i=1}^m U_i \ge s$,  Lemma \ref{lem:bounded-sum}  enables us to  choose  integers $c_1, \ldots, c_m$ such that  $0 \le c_i \le U_i$ 
  for all $1\leq i \leq m$,  and $\sum_{i=1}^m c_i = s$. 
 In particular,  $c_i \le U_i=\lfloor \alpha_i/m\rfloor$ gives that  $\alpha_i\geq mc_i$ for all $1\leq i \leq m$. 
Define,    for all $1\leq i \leq m$,  $\gamma_i := mc_i,$  $\beta_i := \alpha_i - \gamma_i,$ and set 
 $$b := \prod_{i=1}^m x_i^{\gamma_i} \in J^s  \;\; \text{and} \;\; a := \prod_{i=1}^m x_i^{\beta_i} \cdot x_{m+1}^{\alpha_{m+1}} \cdots x_n^{\alpha_n}.$$
 Then $u = ab$. Moreover, we have the following 
  $$\sum_{i=1}^m \beta_i = |\alpha| - m\sum_{i=1}^m c_i = |\alpha| - ms \ge m(r+s) - ms = mr.$$
 Hence,   $a \in ({\mathfrak{p}}^m)^r$, and so   $u \in ({\mathfrak{p}}^m)^r J^s$. This implies 
$(\mathfrak{p}^m)^{r+s} \cap J^s \subseteq (\mathfrak{p}^m)^r J^s$, and thus $J$ is a demotion of $\mathfrak{p}^m$.
   Finally, since $J \subsetneq {\mathfrak{p}}^m$ when $m > 1$, this shows that $J$ is a proper demotion of ${\mathfrak{p}}^m$. 
\end{proof}


This section is concluded with the following proposition, which addresses the demotion of principal monomial ideals. 
 
 \begin{proposition} \label{Demotion-2}
Suppose that  $I = (m) \subset R = K[x_1, \dots, x_n]$ is  a principal monomial ideal, where $m = x_1^{a_1} x_2^{a_2} \cdots x_n^{a_n}$.  
Then a monomial ideal $J \subseteq I$ is a \emph{demotion} of $I$ if and only if  $J= (m u_1, \dots, m u_t)$, where $\mathcal{G}(J)=\{mu_1, \dots, m u_t\}$
 with  $\mathrm{supp}(u_i)\cap \mathrm{supp}(m)=\emptyset$ for all $i=1, \ldots, t$.
  In particular, $I$ admits proper demotions whenever $\operatorname{supp}(m) \subsetneq \{1, \dots, n\}$.
\end{proposition}

\begin{proof}
 ($\Rightarrow$)  Suppose $J \subseteq I$ is a demotion of $I$. Since $I$ is principal, every generator of $J$ is divisible by $m$, so we can write
 $J= (m u_1, \dots, m u_t)$, where $\mathcal{G}(J)=\{mu_1, \dots, m u_t\}$.  We claim that  $\mathrm{supp}(u_i)\cap \mathrm{supp}(m)=\emptyset$ 
for all $i=1, \ldots, t$.  Suppose, on the contrary, that  there exists some $1\leq i \leq t$ such that  $x_j\in \mathrm{supp}(u_i)\cap \mathrm{supp}(m)$
for some $1\leq j \leq n$. Among all such generators that involve $x_j$ in their supports, assume that $\deg_{x_j}u_i$ is minimum. 
By virtue of $J$ is a demotion of $I$, we have $I^{r+s}\cap J^s=I^rJ^s$ for all $r,s\geq 0$. 
Take $r = s = 1$. Then we get $I J = I^2 \cap J.$ Define $M:=\mathrm{lcm}(m^2, m u_i) \in I^2 \cap J.$ Hence, $M\in IJ$. 
This implies that there exists some monomial $m^2u_k\in IJ$, where $1\leq k\leq t$,  such that $m^2u_k\mid M$. As    $x_j\in \mathrm{supp}(u_i)\cap \mathrm{supp}(m)$, we can assume that   $\deg_{x_j}(u_i) = b_j > 0$ and   $a_j>0$.   Also, let  $\deg_{x_j}(u_k) = c_j \geq 0$.  Hence, we may consider the following cases:

\textbf{Case 1.} $c_j >0$. Then the minimality implies that $b_j\leq c_j$, and so we can conclude the  following 
$$\deg_{x_j}(M) = \max\{2a_j, a_j + b_j\} < 2a_j + b_j \leq 2a_j + c_j=\deg_{x_j}(m^2 u_k).$$
This yields that  $M$ is not divisible by $m^2 u_k$, a contradiction.

 \textbf{Case 2.} $c_j = 0$. Hence, $x_j\nmid u_k$. Since $mu_i,mu_k\in \mathcal{G}(J)$ and $u_i\neq u_k$, this implies that there exists some $x_p$ such that 
 $x_p\mid u_k$ but $x_p\nmid u_i$. In particular, $x_p\neq x_j$.  Then we obtain the following 
$$\deg_{x_p}(M) = 2a_p < 2a_p +\deg_{x_p}u_k=\deg_{x_j}(m^2 u_k).$$
This leads  to $m^2u_k\nmid M$, which is a contradiction. 
Consequently, we can conclude that $\mathrm{supp}(u_i)\cap \mathrm{supp}(m)=\emptyset$ for all $i=1, \ldots, t$. 

($\Leftarrow$)  Conversely, let   $\mathcal{G}(J)=\{mu_1, \dots, m u_t\}$ with  $\mathrm{supp}(u_i)\cap \mathrm{supp}(m)=\emptyset$
 for all $i=1, \ldots, t$.   Then any product in $I^r J^s$ has the form $m^{r+s}(u_{i_1} \cdots u_{i_s}).$ Since the variables of $m$ and the $u_i$ 
 are disjoint, for any choice of indices we have  $\operatorname{lcm}\bigl(m^{r+s},\, m^s(u_{i_1} \cdots u_{i_s})\bigr)= m^{r+s}(u_{i_1} \cdots u_{i_s}).$   
Thus, for all $r,s \geq 0$, $I^r J^s = I^{r+s} \cap J^s,$ so $J$ is indeed a demotion of $I$.

Finally, we observe  that if  $\operatorname{supp}(m) = \{1, \dots, n\}$, then no variable is left outside the support of $m$, so no nontrivial $u_i \neq 1$ exists and the only demotion is $J = I$. On the other hand, if $\operatorname{supp}(m) \subsetneq \{1, \dots, n\},$ then there exist variables not appearing in $m$. For any choice of monomials $u_1, \dots, u_t$ in these complementary variables, the ideal $J = (m u_1, \dots, m u_t) \subsetneq I$ is a proper demotion of $I$. 
\end{proof}


\begin{corollary}
any proper monomial ideal in $K[x]$  has no proper demotion ideal.
\end{corollary}


\section{Demotions of  monomial  ideals under   monomial operations}  \label{Section 5}

In this section, we investigate the behavior of the demotions of monomial ideals under a variety of monomial operations. In particular, we explore how these ideals are affected by operations such as expansion, weighting, monomial multiplication, monomial localization, contraction, and deletion. Building upon these operations, we propose several systematic methods for constructing new monomial ideals whose demotions are determined by those of previously studied monomial ideals. To provide a rigorous foundation for our discussion, we begin by examining the demotions of monomial ideals under the sum of ideals.

\subsection{Demotions of  monomial  ideals  under summation}  $\newline$ 

Proposition \ref{Summation} asserts that, under certain condition, the sum of two monomial ideals that possess demotions also admits a demotion. Before proceeding, we first mention the following auxiliary lemma.

\begin{lemma} \label{LEM.SUM1}
Let $J_1 \subseteq I_1$ (respectively,  $J_2 \subseteq I_2$)  be monomial ideals  in $R=K[x_1, \ldots, x_n]$ such that 
$\mathcal{G}(I_1) \subset R_1=K[x_1, \ldots, x_m]$ and $\mathcal{G}(I_2) \subset R_2=K[x_{m+1}, \ldots, x_n]$ for some $m \geq 1$. 
Let $r,s,c,d\geq 0$ with $c+d=s$. Then 
\[
\Big(\sum_{\alpha+\beta=r+s} I_1^\alpha I_2^\beta\Big)\cap (J_1^cJ_2^d)
=
\sum_{\substack{a+b=r+s\\ a\ge c,\ b\ge d}}
\big(I_1^{a}\cap J_1^c\big)\,\big(I_2^{b}\cap J_2^d\big).
\]
\end{lemma}

\begin{proof}
Our strategy is to  show that each side is contained in the other.

\subsubsection*{Claim 1: The right-hand side is contained in the left-hand side.}

Let \( u \) be a monomial in  the right-hand side. Then for some \( a, b \) with \( a + b = r + s \), \( a \geq c \), \( b \geq d \), we have 
$u \in (I_1^a \cap J_1^c)(I_2^b \cap J_2^d)$. So \( u= u_1 u_2 \) with \( u_1 \in I_1^a \cap J_1^c \) and \( u_2 \in I_2^b \cap J_2^d \). Clearly, \( u \in I_1^a I_2^b \subseteq \sum_{\alpha + \beta = r + s} I_1^\alpha I_2^\beta \). Also, \( u_1 \in J_1^c \) and \( u_2 \in J_2^d \) imply that  \( u \in J_1^c J_2^d \). Therefore, \( u \) is in the left-hand side.

\subsubsection*{Claim 2:  The left-hand side is contained in the right-hand side.}

Let \( f \) be a monomial in  the left-hand side. Then \( f \in J_1^c J_2^d \) and \( f \in \sum_{\alpha + \beta = r + s} I_1^\alpha I_2^\beta \). 
Since the variables of \( R_1 \) and \( R_2 \) are disjoint, any monomial \( f \) can be uniquely written as \( f = f_1 f_2 \) with \( f_1 \in R_1 \) and \( f_2 \in R_2 \). Because \( f \in J_1^c J_2^d \), we have \( f_1 \in J_1^c \) and \( f_2 \in J_2^d \). Also, since \( f \in \sum_{\alpha + \beta = r + s} I_1^\alpha I_2^\beta \), there exist \( \alpha, \beta \) with \( \alpha + \beta = r + s \) such that \( f \in I_1^\alpha I_2^\beta \). Hence, \( f_1 \in I_1^\alpha \) and \( f_2 \in I_2^\beta \). 
Therefore, $f_1 \in I_1^\alpha \cap J_1^c$ and $f_2 \in I_2^\beta \cap J_2^d$. 
If \( \alpha \geq c \) and \( \beta \geq d \), then directly \( f \in (I_1^\alpha \cap J_1^c)(I_2^\beta \cap J_2^d) \), which is a term in the right-hand sum.

Now suppose \( \alpha < c \) or \( \beta < d \). Without loss of generality, assume \( \alpha < c \). Clearly,   $f \in (I_1^\alpha \cap J_1^c)(I_2^\beta \cap J_2^d) \subseteq J_1^c J_2^d.$ 
Since \( \alpha + \beta = r + s \) and \( \alpha < c \), we have \( \beta = r + s - \alpha > r + s - c \). As \( c + d \leq r + s \), it follows that \( r + s - c \geq d \), so \( \beta \geq d + 1 \geq d \). Now choose \( a:= c \) and \( b:= r + s - c \). Note that \( b \geq d \). Then, 
$f_1 \in I_1^\alpha \cap J_1^c \subseteq J_1^c = I_1^c \cap J_1^c$ and $f_2 \in I_2^\beta \cap J_2^d \subseteq I_2^\beta \subseteq I_2^{r + s - c} = I_2^b$ 
by  \( \beta >b\). Also, \(f_2 \in J_2^d \), so \( f_2 \in I_2^b \cap J_2^d \). Therefore, we obtain 
\[
f=f_1f_2 \in (I_1^c \cap J_1^c)(I_2^b \cap J_2^d),
\]
which is a term in the right-hand sum  for \( a = c \)  and  \( b = r + s - c \). 

The case \( \beta < d \) is symmetric. Thus, every element of the left-hand side is contained in the right-hand side. This finishes the proof. 
\end{proof}


Recall that if $I$ and $J$ are two monomial ideals in $R=K[x_1, \ldots, x_n]$ such that $\mathrm{supp}(I) \cap \mathrm{supp}(J)= \emptyset$, then 
 $I\cap J=IJ$. 

\begin{proposition} \label{Summation}
 Let $J_1 \subseteq I_1$ (respectively,  $J_2 \subseteq I_2$)  be monomial ideals  in $R=K[x_1, \ldots, x_n]$ such that 
$\mathcal{G}(I_1) \subset R_1=K[x_1, \ldots, x_m]$ and $\mathcal{G}(I_2) \subset R_2=K[x_{m+1}, \ldots, x_n]$ for some $m \geq 1$. 
 If  $J_1$ is  a demotion of $I_1$ (respectively,  $J_2$ is  a demotion of $I_2$), then   $J_1+J_2$ is  a demotion of $I_1+I_2$. 
\end{proposition}
\begin{proof}
Assume that   $J_1$ is  a demotion of $I_1$ (respectively,  $J_2$ is  a demotion of $I_2$). For simplicity of notation, let us define $J := J_1 + J_2$ 
and $I := I_1 + I_2$.  Fix $r,s\geq 0$.  Since the variables in \(R_1\) and \(R_2\) are disjoint, the polynomial
 ring decomposes as a tensor product $R = R_1 \otimes_K R_2.$ It is obvious that $J \subseteq I$. 
 Moreover, based on the binomial expansion theorem,  it follows that
\[
I^r = \sum_{a+b=r} I_1^a I_2^b \quad \text{and}  \quad J^s = \sum_{c+d=s} J_1^c J_2^d,
\]
where multiplication is taken inside \(R\).  By distributivity, we have 
\[
I^r J^s = \left( \sum_{a+b=r} I_1^a I_2^b \right) \left( \sum_{c+d=s} J_1^c J_2^d \right) = \sum_{\substack{a+b=r, \; c+d=s}} I_1^a I_2^b J_1^c J_2^d.
\]

Because the variables of \(R_1\) and \(R_2\) are disjoint, these factors commute and can be grouped as $I_1^a J_1^c \cdot I_2^b J_2^d \subseteq R_1 \otimes_K R_2.$
Due to  \(J_1\) is a demotion of \(I_1\), we get $I_1^a J_1^c = I_1^{a+c} \cap J_1^c$ for all \(a, c \geq 0\). Similarly, since \(J_2\) is a demotion of \(I_2\), we have
$I_2^b J_2^d = I_2^{b+d} \cap J_2^d$ for all \(b, d \geq 0\). Substituting these back,  
\[
I^r J^s = \sum_{\substack{a+b=r, \; c+d=s}} (I_1^{a+c} \cap J_1^c)(I_2^{b+d} \cap J_2^d).
\]

By virtue of  the variables are disjoint, we can readily deduce that 
 \begin{align*}
  (I_1^{a+c} I_2^{b+d}) \cap (J_1^c J_2^d) &=   (I_1^{a+c} \cap  I_2^{b+d}) \cap (J_1^c \cap  J_2^d)\\
 &= (I_1^{a+c} \cap J_1^c) \cap ( I_2^{b+d} \cap  J_2^d)\\
 &= (I_1^{a+c} \cap J_1^c) ( I_2^{b+d} \cap  J_2^d).
 \end{align*}
Thus, we can conclude the following equality 
\[
I^r J^s = \sum_{\substack{a+b=r,\; c+d=s}} (I_1^{a+c} I_2^{b+d}) \cap (J_1^c J_2^d).
\]

 It follows now from Lemma \ref{LEM.SUM1} that 
\[
\begin{aligned}
I^{r+s}\cap J^s
&=\Big(\sum_{\alpha+\beta=r+s} I_1^\alpha I_2^\beta\Big)\cap
\Big(\sum_{c+d=s} J_1^cJ_2^d\Big) \\
&= \sum_{c+d=s}
\Big[\Big(\sum_{\alpha+\beta=r+s} I_1^\alpha I_2^\beta\Big)\cap (J_1^cJ_2^d)\Big] \\
&= \sum_{\substack{c+d=s,\; a+b=r+s\\ a\ge c,\ b\ge d}}
\big(I_1^{a}\cap J_1^c\big)\,\big(I_2^{b}\cap J_2^d\big).
\end{aligned}
\]
On account of  $J_1$ is a demotion of $I_1$ and $J_2$ is a demotion of $I_2$, for $a\ge c$ and $b\ge d$,  we have
$I_1^{a}\cap J_1^c = I_1^{a-c}J_1^c$ and $I_2^{b}\cap J_2^d = I_2^{b-d}J_2^d$. 
Thus, we obtain
\[
I^{r+s}\cap J^s = \sum_{\substack{c+d=s,\; a+b=r+s\\ a\ge c,\ b\ge d}}
I_1^{a-c}J_1^c \; I_2^{b-d}J_2^d.
\]
Let $a':=a-c$ and $b':=b-d$. Then $a'+b'=(a+b)-(c+d)=(r+s)-s=r$, and so we get the following equalities 
\[
I^{r+s}\cap J^s = \sum_{\substack{a'+b'=r, \; c+d=s}}
I_1^{a'}I_2^{b'}  J_1^cJ_2^d
=\Big(\sum_{a'+b'=r} I_1^{a'}I_2^{b'}\Big)
\Big(\sum_{c+d=s} J_1^cJ_2^d\Big)
= I^rJ^s.
\]
We can at once  conclude that   \(J = J_1 + J_2\) is a demotion of \(I = I_1 + I_2\), as claimed. 
\end{proof}
 

\subsection{Demotions of  monomial  ideals  under monomial localization, contraction,  and   deletion}  $\newline$ 

We now  focus on the notion of monomial localization. Assume that  $I \subset R = K[x_1, \ldots, x_n]$ is  a monomial ideal, where $K$ is a field. 
We denote by $V^*(I)$ the collection of monomial prime ideals that contain $I$. Suppose  that  $\mathfrak{p} = (x_{i_1}, \ldots, x_{i_r})$ is a monomial prime ideal. 
The \emph{monomial localization} of $I$ with respect to $\mathfrak{p}$, denoted by $I(\mathfrak{p})$, is defined as the ideal of the polynomial ring
 $R(\mathfrak{p}) = K[x_{i_1}, \ldots, x_{i_r}],$ obtained from $I$ via the $K$-algebra homomorphism
\[
R \longrightarrow R(\mathfrak{p}), \qquad x_j \mapsto 1 \quad \text{for all } x_j \notin \{x_{i_1}, \ldots, x_{i_r}\}.
\]

\bigskip
The following lemma will be needed in order to prove Proposition \ref{PRO.Localization}. For the sake of clarity and ease of reference, we state it here.
 
  \begin{lemma}\label{LEM.Localization} (\cite[Lemma 3.13]{SN})
Let $I$ and $J$ be two monomial ideals   in $R=K[x_1, \ldots, x_n]$, and $\mathfrak{p}$ be a monomial prime ideal of $R$. Then the following 
statements hold.
\begin {itemize}
\item[(i)]  $(I+J)(\mathfrak{p})= I(\mathfrak{p}) + J(\mathfrak{p})$;
\item[(ii)] $(IJ)(\mathfrak{p})= I(\mathfrak{p})  J(\mathfrak{p})$;
\item[(iii)] $(I\cap J)(\mathfrak{p})= I(\mathfrak{p}) \cap  J(\mathfrak{p})$;
\item[(iv)] $(I :_RJ)(\mathfrak{p})= (I(\mathfrak{p}) :_{R(\mathfrak{p})}  J(\mathfrak{p}))$;
\item[(v)]  If $Q$ is a $\mathfrak{q}$-primary monomial ideal in $R$ with $ \mathfrak{q}\subseteq \mathfrak{p}$, then $Q(\mathfrak{p})$ is a $\mathfrak{q}$-primary monomial ideal in $R(\mathfrak{p})$.
\end{itemize}
 \end{lemma}


 The following proposition  asserts that if a monomial ideal has the demotion property, then any of its monomial localizations with respect to a monomial prime ideal also possesses this property.

\begin{proposition}\label{PRO.Localization}
Let  $J \subseteq I$ be monomial ideals in  $R=K[x_1, \ldots, x_n]$, and $\mathfrak{p}\in V^*(I)$. If  $J$ is a demotion of $I$, then  $J(\mathfrak{p})$ 
is a demotion of   $I(\mathfrak{p})$. 
\end{proposition}
\begin{proof}
Suppose that $J$ is a demotion of $I$,   and fix $r,s\geq 0$. Then we have $I^rJ^s=I^{r+s} \cap J^s$, and so 
  $(I^rJ^s)(\mathfrak{p})=(I^{r+s} \cap J^s)(\mathfrak{p})$. It follows from  Lemma \ref{LEM.Localization}(i) that  $J(\mathfrak{p}) \subseteq I(\mathfrak{p})$. 
   We can now derive from parts (ii) and (iii) of  Lemma \ref{LEM.Localization}  that 
 $I(\mathfrak{p})^rJ(\mathfrak{p})^s=I(\mathfrak{p})^{r+s} \cap J(\mathfrak{p})^s$. We therefore obtain  $J(\mathfrak{p})$ 
is a demotion of   $I(\mathfrak{p})$, as desired. 
 \end{proof}


To state the forthcoming result, we first recall the definition of the contraction operation. Let $I$ be a monomial ideal in $R = K[x_1, \ldots, x_n]$ with minimal generating set $\mathcal{G}(I) =\{u_1, \ldots, u_m\}$. For some index $1 \leq j \leq n$, the \emph{contraction} of $x_j$ from $I$, denoted by $I / x_j$, is obtained by substituting $x_j = 1$ in each generator $u_i$, for $i = 1, \ldots, m$.
Since the contraction $I / x_j$ coincides with the monomial localization of $I$ with respect to $\mathfrak{p} = \mathfrak{m} \setminus \{x_j\}$, where $\mathfrak{m} = (x_1, \ldots, x_n)$ denotes the graded maximal ideal of $R$, Proposition \ref{PRO.Localization} allows us to immediately deduce the following result.
 
 \begin{corollary}\label{Cor. contraction}
 Let  $J \subseteq I$ be monomial ideals in  $R=K[x_1, \ldots, x_n]$, and     $1\leq j \leq n$.  If  $J$ is a demotion of $I$, then  $J/x_j$ 
is a demotion of    $I/x_j$. 
 \end{corollary}


We continue by recalling the definition of the deletion operation. Let $I \subset R = K[x_1, \ldots, x_n]$ be a monomial ideal with minimal 
generating set $\mathcal{G}(I) = \{u_1, \ldots, u_m\}$. For some index $1 \leq j \leq n$, the \emph{deletion} of $x_j$ from $I$, denoted by 
$I \setminus \{x_j\}$ (sometimes written as $I \setminus x_j$ for simplicity), is obtained by setting $x_j = 0$ in each generator $u_i$, for all $i = 1, \ldots, m$.

The lemma below plays a crucial role in the proof of Proposition \ref{PRO.Deletion}. For the reader’s convenience, we include it here.

\begin{lemma}(\cite[Lemma 4.12]{RNA})\label{LEM.Deletion}
Let $I$ and $J$ be two monomial ideals   in  $R=K[x_1,\ldots, x_n]$ over a field $K$,  and $1\leq j \leq n$. Then the following statements hold.
\begin {itemize}
\item[(i)] $(I+J)\setminus x_j= I\setminus x_j + J\setminus x_j$;
\item[(ii)] $(IJ)\setminus x_j= (I\setminus x_j)  (J\setminus x_j)$;
\item[(iii)] $(I\cap J)\setminus x_j= (I\setminus x_j) \cap  (J\setminus x_j)$; 
\item[(iv)]  $\sqrt{I\setminus x_j}=\sqrt{I}\setminus x_j$.
\item[(v)] If $Q$ is a $\mathfrak{q}$-primary monomial ideal in $R$, then $Q\setminus x_j$ is a ($\mathfrak{q}\setminus x_j$)-primary monomial ideal in $R\setminus x_j$;
\item[(vi)] If $I=Q_1\cap \cdots \cap Q_s$ is a minimal primary decomposition of $I$ in $R$, then  $I\setminus x_j=(Q_1\setminus x_j)\cap \cdots \cap (Q_s\setminus x_j)$ is a  primary decomposition of $I\setminus x_j$ in $R\setminus x_j$;
\item[(vii)] $\mathrm{Ass}_{R\setminus x_j}((R \setminus x_j)/(I\setminus x_j))\subseteq \{\mathfrak{q}\setminus x_j~:~ \mathfrak{q}\in \mathrm{Ass}_{R}(R/I) \}$.
\end{itemize}
 \end{lemma}


The following proposition states that whenever a monomial ideal satisfies the demotion property, each of its deletions also retains this property. 

\begin{proposition}\label{PRO.Deletion}
Let  $J \subseteq I$ be monomial ideals in  $R=K[x_1, \ldots, x_n]$, and  $1\leq j \leq n$.  If  $J$ is a demotion of $I$, then  $J\setminus x_j$ 
is a demotion of   $I\setminus x_j$. 
\end{proposition}
\begin{proof}
Assume that $J$ is a demotion of $I$,   and fix $r,s\geq 0$. This implies that  $I^rJ^s=I^{r+s} \cap J^s$, and hence
  $(I^rJ^s)\setminus x_j=(I^{r+s} \cap J^s)\setminus x_j$. According to  Lemma \ref{LEM.Deletion}(i), we obtain  $J\setminus x_j \subseteq I\setminus x_j$. 
   It can be concluded  from parts (ii) and (iii) of  Lemma \ref{LEM.Deletion} that  
   $(I\setminus x_j)^r (J\setminus x_j)^s=(I\setminus x_j)^{r+s} \cap (J\setminus x_j)^s$. This gives rise to $J\setminus x_j$ 
is a demotion of   $I\setminus x_j$, as claimed. 
 \end{proof}


\subsection{Demotions of  monomial  ideals under permutation  and monomial  multiple}$\newline$

 In this subsection, we investigate how the demotion property behaves under permutations and the monomial multiple operation in Propositions \ref{PRO.Permutation} and \ref{Multiple}, respectively.  Prior to presenting Proposition \ref{PRO.Permutation}, we first need to  recall the following definition.

 \begin{definition} (\cite[Definition 11]{RT}) \label{Permutation}
 \em{
  Let $I\subset R=K[x_1, \ldots, x_n]$ be a square-free monomial ideal and   $\sigma$ a permutation on $\mathrm{supp}(I)$.
  Then we define the $\sigma(I)$ as the square-free monomial ideal in $R$ such that 
  $x_{i_1} \cdots x_{i_s}\in \mathcal{G}(I)$ if and only if   $x_{\sigma(i_1)} \cdots x_{\sigma(i_s)}\in \mathcal{G}(\sigma(I))$.  
  }
 \end{definition}

\begin{proposition} \label{PRO.Permutation}
With the notation of Definition \ref{Permutation}, a square-free monomial ideal $J$ is a demotion of a square-free monomial ideal $I$ if and only if $\sigma(J)$ is a demotion of $\sigma(I)$. 
\end{proposition}

\begin{proof}
  Let  $\mathrm{supp}(I)=\{x_1, \ldots, x_n\}$, and fix $r,s\geq 0$.  Consider the $K$-algebra homomorphism $\varphi : R=K[x_1, \ldots, x_n] \rightarrow R$  given by $\varphi(x_i)=x_{\sigma(i)}$ for all $i=1, \ldots, n$.  It is straightforward to see that  $\varphi$ is an automorphism of $R$ with  $\varphi(I)=\sigma(I)$. Particularly, it is clear that $J\subseteq I$ if and only if $\sigma(J) \subseteq \sigma(I)$. 
  
(\(\Rightarrow\))
Assume \(I^{r+s} \cap J^s = I^r J^s\) holds. We want to show that 
\[
(\sigma(I))^{r+s} \cap (\sigma(J))^s = (\sigma(I))^r (\sigma(J))^s.
\]
Apply \(\varphi\) to both sides of \(I^{r+s} \cap J^s = I^r J^s\), we get $\varphi(I^{r+s} \cap J^s) = \varphi(I^r J^s).$ 
 Due to  \(\varphi\) is an automorphism of $R$, we have the following equalities 
         \[
        \varphi(I^{r+s} \cap J^s) = \varphi(I^{r+s}) \cap \varphi(J^s) = (\sigma(I))^{r+s} \cap (\sigma(J))^s, \text{and}
        \]
        \[
        \varphi(I^r J^s) = \varphi(I^r) \varphi(J^s) = (\sigma(I))^r (\sigma(J))^s.
        \]

Thus, we can conclude that $(\sigma(I))^{r+s} \cap (\sigma(J))^s = (\sigma(I))^r (\sigma(J))^s.$ 

(\(\Leftarrow\))  Assume $(\sigma(I))^{r+s} \cap (\sigma(J))^s = (\sigma(I))^r (\sigma(J))^s.$ Our aim is to establish \(I^{r+s} \cap J^s = I^r J^s\).  
As \(\varphi\) is an isomorphism, it has an inverse \(\varphi^{-1}\). Apply \(\varphi^{-1}\) to both sides, we obtain 
$\varphi^{-1}\left((\sigma(I))^{r+s} \cap (\sigma(J))^s\right) = \varphi^{-1}\left((\sigma(I))^r (\sigma(J))^s\right).$
This implies that  $\varphi^{-1}((\sigma(I))^{r+s}) \cap \varphi^{-1}((\sigma(J))^s) = I^{r+s} \cap J^s$ and moreover
    $\varphi^{-1}((\sigma(I))^r) \varphi^{-1}((\sigma(J))^s) = I^r J^s.$ Hence, $I^{r+s} \cap J^s = I^r J^s.$
\end{proof}


\begin{proposition}\label{Multiple}
Let  $J \subseteq I$ be monomial ideals in  $R=K[x_1, \ldots, x_n]$, and   $h$ be a monomial in $R$ such that 
$\mathrm{supp}(h) \cap (\mathrm{supp}(I) \cup \mathrm{supp}(J))=\emptyset$. 
 Then  $J$ is a demotion of $I$ if and only if   $hJ$  is a demotion of   $hI$. 
  \end{proposition}
  
\begin{proof}
Since $K$ is a field, we obtain  $R=K[x_1, \ldots, x_n]$ is an integral domain,  and so it is easy to see that  $J\subseteq I$ if and only if $hJ \subseteq hI$. 
Fix $r,s\geq 0$. Due to 
$\mathrm{supp}(h) \cap (\mathrm{supp}(I) \cup \mathrm{supp}(J))=\emptyset$, we get the following equalities 
\begin{align*}
h^{r+s}(I^{r+s} \cap J^s)&= (h^{r+s}) \cap (I^{r+s} \cap J^s)\\
 &=((h)^{r+s} \cap I^{r+s}) \cap ((h)^s  \cap J^s)\\
 &=(hI)^{r+s} \cap (hJ)^s. 
 \end{align*}
On the other hand, we always have $(hI)^r (hJ)^s = h^{r+s}I^rJ^s$.  On account of $R=K[x_1, \ldots, x_n]$ is an integral domain, we can rapidly deduce that 
$I^{r+s} \cap J^s=I^rJ^s$ if and only if $(hI)^{r+s} \cap (hJ)^s=(hI)^r(hJ)^s$.  This means that $J$ is a demotion of $I$ if and only if   $hJ$  is a demotion of   $hI$. 
\end{proof}


\subsection{Demotions of  monomial  ideals under expansion}$\newline$

We begin this subsection by recalling the definition of the expansion operation on monomial ideals, as introduced in \cite{BH}.  

Let $K$ be a field and consider the polynomial ring $R = K[x_1, \ldots, x_n]$ over $K$ in the variables $x_1, \ldots, x_n$. Fix an ordered $n$-tuple $(i_1, \ldots, i_n)$ of positive integers, and define the polynomial ring  $R^{(i_1,\ldots,i_n)}$ over $K$ with variables
\[
x_{11}, \ldots, x_{1i_1}, \; x_{21}, \ldots, x_{2i_2}, \; \ldots, \; x_{n1}, \ldots, x_{ni_n}.
\]

For each $j=1,\ldots,n$, let $\mathfrak{p}_j$ denote the monomial prime ideal
\[
\mathfrak{p}_j = (x_{j1}, x_{j2}, \ldots, x_{ji_j}) \subseteq R^{(i_1,\ldots,i_n)}.
\]

Now, let $I \subset R$ be a monomial ideal with a minimal generating set $\mathcal{G}(I) = \{\mathbf{x}^{\mathbf{a}_1}, \ldots, \mathbf{x}^{\mathbf{a}_m}\},$ 
 where  $\mathbf{x}^{\mathbf{a}_i} = x_1^{a_i(1)} \cdots x_n^{a_i(n)},$ 
and $\mathbf{a}_i = (a_i(1), \ldots, a_i(n))$ with $a_i(j)$ denoting the $j$th component of $\mathbf{a}_i$ for $i = 1,\ldots,m$.  

The \emph{expansion of $I$ with respect to the $n$-tuple $(i_1, \ldots, i_n)$}, denoted by $I^{(i_1,\ldots,i_n)}$, is defined as the monomial ideal
\[
I^{(i_1,\ldots,i_n)} = \sum_{i=1}^m \mathfrak{p}_1^{a_i(1)} \cdots \mathfrak{p}_n^{a_i(n)} \subseteq R^{(i_1,\ldots,i_n)}.
\]
For simplicity, we will often write $R^*$ and $I^*$ instead of $R^{(i_1,\ldots,i_n)}$ and $I^{(i_1,\ldots,i_n)}$, respectively.\par 
For example, assume that $R = K[x_1, x_2, x_3, x_4]$ and consider the ordered $4$-tuple $(2,3,1,2)$. We then obtain 
\[
\mathfrak{p}_1 = (x_{11}, x_{12}), \quad 
\mathfrak{p}_2 = (x_{21}, x_{22}, x_{23}), \quad
\mathfrak{p}_3 = (x_{31}), \quad
\mathfrak{p}_4 = (x_{41}, x_{42}).
\]
Now, for the monomial ideal  $I = (x_1^2x_2,x_1x_3,x_2x_4^2) \subset R,$  
the expanded ideal $I^* \subseteq K[x_{11}, x_{12}, x_{21}, x_{22}, x_{23}, x_{31}, x_{41}, x_{42}]$ is given by 
 $I^* = \mathfrak{p}_1^2 \mathfrak{p}_2 + \mathfrak{p}_1 \mathfrak{p}_3 + \mathfrak{p}_2 \mathfrak{p}_4^2,$ namely,
\begin{align*}
I^* = (&x_{11}^2 x_{21},\, x_{11}^2 x_{22},\, x_{11}^2 x_{23},\, x_{12}^2 x_{21},\, x_{12}^2 x_{22},\, x_{12}^2 x_{23},\\
& x_{11} x_{31},\, x_{12} x_{31},\, x_{21}x_{41}^2,\, x_{21}x_{41}x_{42},\, x_{21}x_{42}^2,\, x_{22}x_{41}^2,\\
& x_{22}x_{41}x_{42},\, x_{22}x_{42}^2,\, x_{23}x_{41}^2,\, x_{23}x_{41}x_{42},\, x_{23}x_{42}^2 ).
\end{align*}

\bigskip
To establish  Proposition \ref{PRO.Expansion}, we first make use of the following auxiliary result.

\begin{lemma} (\cite [Lemma 1.1]{BH}) \label{Lem.Bayati.Expansion}
Let $I$ and $J$ be monomial ideals in a polynomial ring $S$. Then
\begin{itemize}
\item[(i)] $f \in I^*$ if and only if $\pi(f)\in I$, for all $f\in S^*$;
\item[(ii)] $(I + J)^* = I^* + J^*$;
\item[(iii)] $(IJ)^* = I^*J^*$;
\item[(iv)] $(I \cap J)^* = I^*\cap J^*$;
\item[(v)] $(I : J)^* = (I^* : J^*)$;
\item[(vi)] $\sqrt{I^*} = (\sqrt{I})^*$;
\item[(vii)] If the monomial ideal $Q$ is $\mathfrak{p}$-primary, then $Q^*$ is
$\mathfrak{p}^*$-primary.
\end{itemize}
\end{lemma}


 The next proposition states  that a monomial ideal has the demotion  property if and only if its expansion  has the demotion  property.

\begin{proposition}\label{PRO.Expansion} 
Let  $J \subseteq I$ be monomial ideals in  $R=K[x_1, \ldots, x_n]$. Then $J$ is a demotion of $I$  if and only if  $J^*$ is a demotion of $I^*$, 
where $I^*$ denotes the  expansion of $I$.  
\end{proposition}
\begin{proof}
 Suppose that  $J$ is   a demotion of $I$, and fix $r,s\geq 0$. Then we have $I^rJ^s=I^{r+s} \cap J^s$, and so   $(I^rJ^s)^*=(I^{r+s} \cap J^s)^*$. We also  deduce from Lemma \ref{Lem.Bayati.Expansion}(ii) that $J^* \subseteq I^*$.  In light of parts (iii) and (iv) of   Lemma \ref{Lem.Bayati.Expansion}, one can derive   that $(I^*)^r(J^*)^s=(I^*)^{r+s} \cap (J^*)^s$. Accordingly,  $J^*$ is a demotion of $I^*$,  as required. 
  
 The converse can be shown  by a similar argument and recalling this fact that for any two monomial  ideals $L_1, L_2$, we always  have $L_1=L_2$ if and only if 
 $L^*_1=L^*_2$.
 \end{proof}


\subsection{Demotions of  monomial  ideals under weighting} $\newline$

This subsection is devoted to analyzing the demotion property in the context of the weighting operation. For this purpose, we first revisit the relevant definition.

\begin{definition}
\em{
A \textit{weight} over a polynomial ring $R=K[x_1, \ldots, x_n]$ is a function $W: \{x_1, \ldots, x_n\} \rightarrow \mathbb{N}$ such that  $w_i=W(x_i)$ is called the {\it weight} of the variable $x_i$. For a monomial ideal $I\subset R $ and a weight $W$, we define the \textit{weighted ideal}, denoted 
by $I_W$, to be the ideal generated by $\{h(u) : u\in \mathcal{G}(I) \}$, where $h$ is the unique homomorphism $h: R \rightarrow R$ given by 
$h(x_i)= x_i^{w_i}$. 
}
\end{definition}


As an example,  let  $R = K[x_1, x_2, x_3, x_4]$ and consider the monomial ideal $I = (x_1^2x_2,\; x_2^3x_3^2,\; x_1x_3x_4^2) \subset R.$ 
 Also, let $W : \{x_1, x_2, x_3, x_4\} \rightarrow \mathbb{N}$  be a weight over $R$ defined by 
$W(x_1)=2,   W(x_2)=3,  W(x_3)=1,$ and $W(x_4)=4.$  Then, the weighted ideal $I_W$ is given by 
$I_W = (x_1^4x_2^3,\; x_2^9x_3^2,\; x_1^2x_3x_4^8).$

\bigskip
The proof of  Proposition \ref{PRO.Weighting} relies on the application of the following lemma.  

\begin{lemma}\label{LEM. Weighted} (\cite[Lemma 3.5]{SN})
Let $I$ and $J$ be two monomial ideals of a polynomial ring $R=K[x_1, \ldots, x_n]$, and $W$ a weight over $R$. Then the following statements hold. 
\begin {itemize}
\item[(i)] $(I+J)_W= I_W + J_W$;
\item[(ii)] $(IJ)_W= I_W J_W$;
\item[(iii)] $(I\cap J)_W= I_W \cap J_W$;
\item[(iv)] $(I :_RJ)_W= (I_W :_R J_W)$.
\end{itemize}
\end{lemma}

 
 A monomial ideal has the demotion property if and only if its weighted ideal does, as stated in the following proposition.

\begin{proposition}\label{PRO.Weighting}
Let $I \subset R=K[x_1, \ldots, x_n]$  be a  monomial ideal,  and $W$ a weight over $R$.   Then $J$ is a demotion of $I$  if and only if 
 $J_W$ is a demotion of $I_W$, where $I_W$ denotes the weighted ideal.  
\end{proposition}
\begin{proof}
Assume that   $J$ is   a demotion of $I$, and fix $r,s\geq 0$. We thus get  $I^rJ^s=I^{r+s} \cap J^s$, and hence   $(I^rJ^s)_W=(I^{r+s} \cap J^s)_W$. 
On account of  Lemma  \ref{LEM. Weighted}(i), we obtain   $J_W \subseteq I_W$.  We can immediately conclude from parts (ii) and (iii) of 
 Lemma  \ref{LEM. Weighted} that   $(I_W)^r(J_W)^s=(I_W)^{r+s} \cap (J_W)^s$. Therefore,  $J_W$ is a demotion of $I_W$, as claimed. 
  
By remembering  this fact that for any two monomial  ideals $L_1, L_2$, we always have $L_1=L_2$ if and only if 
 $(L_1)_W=(L_2)_W$, a similar discussion can be used to  establish the converse.
\end{proof} 


As a consequence of  Propositions \ref{PRO.Expansion} and \ref{PRO.Weighting}, the theorem below ensures that there exist infinitely many monomial ideals in 
 $R=K[x_1, \ldots, x_n]$  such that  have  proper demotion ideals.

\begin{theorem} \label{Th.Infinite}
Let $R=K[x_1, \ldots, x_n]$ be  a polynomial ring in $n$  variables, where $n\geq 2$,  with coefficients in a  field $K$. Then 
there exist infinitely many  monomial ideals  in $R$ such that  have  proper demotion ideals. 
\end{theorem}

\begin{proof}
Let $S=K[x,y]$, and consider the monomial  ideals $I=(x)$ and $J=(xy)$ is $S$. Clearly $J \subsetneq I$. We claim   $J$ is a demotion of $I$. Fix $r,s\geq 0$. It is easy to check that  $I^rJ^s=(x^r)(x^sy^s)=(x^{r+s}y^s).$  On the other hand,
$$I^{r+s}\cap J^s=(x^{\,r+s})\cap (x^sy^s)= (\mathrm{lcm}(x^{r+s},x^sy^s))=(x^{r+s}y^s).$$
Thus, $I^rJ^s = I^{r+s}\cap J^s$, and so   $J$ is a proper demotion of $I$.
 Now, let $\mathfrak{p}_1:=(x_{i_1}, \ldots, x_{i_a})$ and  $\mathfrak{p}_2:=(x_{i_{a+1}}, \ldots, x_{i_b})$  with  $\mathrm{supp}(\mathfrak{p}_1) \cap  \mathrm{supp}(\mathfrak{p}_2)=\emptyset$ and $\mathrm{supp}(\mathfrak{p}_1) \cup \mathrm{supp}(\mathfrak{p}_2)=\{x_1, \ldots, x_n\}$. 
 One can promptly deduce from  Proposition  \ref{PRO.Expansion} that $J^*$ is a demotion of $I^*$ in $R=K[x_1, \ldots, x_n]$. Now, assume that  $W$ is 
 a weight over $R=K[x_1, \ldots, x_n]$  with $W(x_i)=\alpha_i$ such that $\alpha_i \geq 1$ for all $i=1, \ldots, n$. In  light of  Proposition  \ref{PRO.Weighting}, we obtain   $(J^*)_W$ is a demotion of $(I^*)_W$. This indicates that there exist infinitely many monomial ideals in $R=K[x_1, \ldots, x_n]$  
 such that have proper demotion ideals, as claimed.  
  \end{proof}


\section{Demotions of monomial ideals and normally torsion-freeness} \label{Section 6}

 The primary aim of this section is to investigate the relationship between demotions of monomial ideals and normally torsion-freeness. 
 In particular, we demonstrate that the demotion property can be used to construct new normally torsion-free monomial ideals, 
   and conversely, that normally torsion-free ideals can inform the generation of ideals with the demotion property.
 To do this, we commence with the following proposition.

  \begin{proposition} \label{Demotion-3}
  Let $I\subset R=K[x_1, \ldots, x_n]$  be a  square-free monomial ideal  and  $\mathfrak{q}$ be  a prime monomial ideal in $R$ such that $\bigcap_{\mathfrak{p}\in \mathrm{Ass}(I)}\mathfrak{p} \cap \mathfrak{q}$  is  a minimal primary decomposition of  
$I \cap \mathfrak{q}$.   Let  $I$ and $I \cap \mathfrak{q}$ be  normally torsion-free.  
Then $I\cap \mathfrak{q}$ is a demotion of $I$. 
  \end{proposition}

\begin{proof}
Set $J:=I\cap \mathfrak{q}$. We must show that $I^rJ^s=I^{r+s} \cap J^s$ for all $r,s\geq 0$. To do this, fix $r,s\geq 0$. Assume that 
$\mathrm{Ass}(I)=\{\mathfrak{p}_1, \ldots, \mathfrak{p}_t\}$. Then  $I=\bigcap_{i=1}^t\mathfrak{p}_i$ 
and $J=\bigcap_{i=1}^{t}\mathfrak{p}_i\cap \mathfrak{q}$,  and hence  $J\subseteq I$. In addition, 
  Facts \ref{proposition 4.3.29}, \ref{exercise 6.1.25} and Theorem \ref{Villarreal1} yield that 
 $I^{r+s}=\bigcap_{i=1}^{t}{\mathfrak{p}^{r+s}_i}$, $I^r=\bigcap_{i=1}^{t}{\mathfrak{p}^r_i}$, and  $J^s=\bigcap_{i=1}^{t}{\mathfrak{p}^s_i}\cap \mathfrak{q}^s$. Here, by using Lemma   \ref{NCH-1}, we obtain the following equalities
 \begin{align*}  
I^{r+s} \cap J^s=& \left(\bigcap_{i=1}^{t}\mathfrak{p}_i^{r+s}\right) \cap  \left(\bigcap_{i=1}^{t}\mathfrak{p}_i^{s}\cap    \mathfrak{q}^{s}\right)\\
=& \left(\bigcap_{i=1}^{t}\mathfrak{p}_i^{r+s}\right) \cap     \mathfrak{q}^{s}\\
=& \left(\bigcap_{i=1}^{t}\mathfrak{p}_i^{r}\right)  \left(\bigcap_{i=1}^{t}\mathfrak{p}_i^{s}  \cap \mathfrak{q}^{s}\right)=I^rJ^s.
 \end{align*}
This gives that $J$ is a demotion of $I$, as claimed. 
\end{proof}
  

\begin{theorem} \label{NTT-1}
Suppose that  $I$ and $J$  are  normally torsion-free square-free monomial ideals in $R=K[x_1, \ldots, x_n]$ satisfying the following
\begin{itemize}
\item[(i)] $J$ is a proper demotion of $I$;
\item[(ii)] if $\mathfrak{q} \in \mathrm{Ass}(R/J)\setminus \mathrm{Ass}(R/I)$, then  $\mathfrak{q} \nsubseteq \mathfrak{p}$ for every 
$\mathfrak{p}\in \mathrm{Ass}(R/I).$ 
\end{itemize}
 Then, for all $r,s \geq 1$, $I^rJ^s$ is normally torsion-free. 
\end{theorem}

\begin{proof}
Fix $r,s \geq 1$.  As  $J$ is a demotion of $I$, we obtain    $I^rJ^s=I^{r+s}\cap J^s$.  Let  $\mathrm{Ass}(R/I)=\{\mathfrak{p}_1, \ldots,   \mathfrak{p}_t\}$ and
 put $\Gamma:=\mathrm{Ass}(R/J)\setminus \mathrm{Ass}(R/I)$. 
 Let   $\mathrm{Ass}(R/J)=\{\mathfrak{q}_1, \ldots,   \mathfrak{q}_m, \mathfrak{q}_{m+1}, \ldots,   \mathfrak{q}_\ell\}$ and 
 $\Gamma=\{\mathfrak{q}_1, \ldots, \mathfrak{q}_m\}.$ 
In what follows, we claim that $\mathrm{Ass}(I^{r+s}\cap J^s)=\mathrm{Min}(I) \cup \Gamma$. 
We can conclude from   Facts \ref{proposition 4.3.29}, \ref{exercise 6.1.25} and Theorem \ref{Villarreal1}  that  $I^{r+s}=\bigcap_{i=1}^{t}{\mathfrak{p}^{r+s}_i}$ 
  and  $J^s=\bigcap_{i=1}^{\ell}\mathfrak{q}_i^s$. Let $\mathfrak{q}\in \mathrm{Ass}(R/J)\setminus  \Gamma$. Then $\mathfrak{q}\in \mathrm{Ass}(R/J)\cap  \mathrm{Ass}(R/I)$, 
  say $\mathfrak{q}=\mathfrak{p}_1$. This implies that $\mathfrak{p}_1^{r+s} \cap \mathfrak{q}^s=\mathfrak{p}_1^{r+s}$. This allows us to obtain  
   $I^{r+s} \cap J^s=\bigcap_{i=1}^{t}{\mathfrak{p}^{r+s}_i} \cap \bigcap_{i=1}^{m}\mathfrak{q}_i^s$. To show our claim, it is enough to establish 
  $\bigcap_{i=1}^{t}{\mathfrak{p}^{r+s}_i} \cap \bigcap_{i=1}^{m}\mathfrak{q}_i^s$ is a minimal primary decomposition of 
   $I^{r+s} \cap J^s$. To achieve this, we have to 
  prove the following statements:
  \begin{itemize}
  \item[(i)]  $\bigcap_{i=1}^{t}{\mathfrak{p}^{r+s}_i}\cap  \bigcap_{i=1,i\neq j}^{m}{\mathfrak{q}^{s}_i} \nsubseteq   \mathfrak{q}_j^s$ 
  for all $j=1, \ldots, m$. 
  \item[(ii)]  $\bigcap_{i=1,i\neq j}^{t}{\mathfrak{p}^{r+s}_i}\cap   \bigcap_{i=1}^{m}{\mathfrak{q}^{s}_i}    \nsubseteq \mathfrak{p}^{r+s}_j$
   for all $j=1, \ldots, t$. 
    \end{itemize}
  
  (i) On the contrary, assume  that  $\bigcap_{i=1}^{t}{\mathfrak{p}^{r+s}_i}\cap  \bigcap_{i=1,i\neq j}^{m}{\mathfrak{q}^{s}_i} \subseteq   \mathfrak{q}_j^s$ 
  for some $1\leq j \leq m$. Hence, we obtain  $\sqrt{\bigcap_{i=1}^{t}{\mathfrak{p}^{r+s}_i}\cap  \bigcap_{i=1,i\neq j}^{m}{\mathfrak{q}^{s}_i}} \subseteq  \sqrt{\mathfrak{q}_j^s}$. This leads to  $\bigcap_{i=1}^{t}{\mathfrak{p}_i}\cap  \bigcap_{i=1,i\neq j}^{m}{\mathfrak{q}_i} \subseteq   \mathfrak{q}_j$. 
  We thus get $\mathfrak{p}_\alpha \subseteq \mathfrak{q}_j$ for some $1\leq \alpha \leq t$, or 
  $\mathfrak{q}_\beta \subseteq \mathfrak{q}_j$ for some $1\leq \beta\neq j \leq m$.   We first assume that 
  $\mathfrak{q}_\beta \subseteq \mathfrak{q}_j$ for some $1\leq \beta\neq j \leq m$. 
   Due to  $J$ is a  square-free monomial ideal, this implies that
     $\mathrm{Ass}(J)=\mathrm{Min}(J)$, and hence  $\mathfrak{q}_\beta \nsubseteq \mathfrak{q}_j$ for every  $1\leq \beta\neq j \leq m$, a contradiction. 
     So, let   $\mathfrak{p}_\alpha \subseteq \mathfrak{q}_j$ for some $1\leq \alpha \leq t$.
    Since $J\subsetneq I \subseteq\mathfrak{p}_\alpha \subseteq \mathfrak{q}_j$ and $\mathfrak{q}_j\in \mathrm{Min}(J)$, we must have
 $\mathfrak{q}_j=\mathfrak{p}_\alpha$. This gives that $\mathfrak{q}_j \in \mathrm{Ass}(R/I)$, which is a contradiction.  Consequently, we have 
    $\bigcap_{i=1}^{t}{\mathfrak{p}^{r+s}_i}\cap  \bigcap_{i=1,i\neq j}^{m}{\mathfrak{q}^{s}_i} \nsubseteq   \mathfrak{q}_j^s$ 
  for all $j=1, \ldots, m$. 
  
  (ii)  Suppose, on the contrary, that  there exists some $1\leq j \leq t$ such that
   $\bigcap_{i=1,i\neq j}^{t}{\mathfrak{p}^{r+s}_i}\cap   \bigcap_{i=1}^{m}{\mathfrak{q}^{s}_i}   
  \subseteq \mathfrak{p}^{r+s}_j$. Without loss of generality,  assume that  $\bigcap_{i=1}^{t-1}{\mathfrak{p}^{r+s}_i}\cap 
  \bigcap_{i=1}^{m}{\mathfrak{q}^{s}_i}   \subseteq \mathfrak{p}^{r+s}_{t}$. 
  Since $\bigcap_{i=1}^{t}{\mathfrak{p}_i}$ is a minimal primary decomposition of $I$, we can derive that 
  $\bigcap_{i=1}^{t-1}{\mathfrak{p}_i}\nsubseteq \mathfrak{p}_{t}$. This yields that there exists some monomial 
  $u\in \bigcap_{i=1}^{t-1}{\mathfrak{p}_i} \setminus  \mathfrak{p}_{t}$. 
  As $(\bigcap_{i=1}^{t-1}{\mathfrak{p}_i})^{r+s} \subseteq \bigcap_{i=1}^{t-1}{\mathfrak{p}_i}^{r+s}$, we have    
  $u^{r+s}\in \bigcap_{i=1}^{t-1}{\mathfrak{p}^{r+s}_i} \setminus  \mathfrak{p}^{r+s}_{t}$. We here claim that 
  $\bigcap_{i=1}^{m}{\mathfrak{q}^{s}_i} \nsubseteq \mathfrak{p}^{r+s}_{t}$. 
  Otherwise, $\bigcap_{i=1}^{m}{\mathfrak{q}^{s}_i}  \subseteq \mathfrak{p}^{r+s}_{t}$. Hence, $\bigcap_{i=1}^{m}{\mathfrak{q}_i}  \subseteq \mathfrak{p}_{t}$. This implies that $\mathfrak{q}_z\subseteq \mathfrak{p}_t$ for some $1\leq z \leq m$. 
     Because $\mathfrak{q}_z\in \Gamma$, so we must have $\mathfrak{q}_z\nsubseteq \mathfrak{p}_t$, a contradiction. 
      Accordingly,   $\bigcap_{i=1}^{m}{\mathfrak{q}^{s}_i} \nsubseteq \mathfrak{p}^{r+s}_{t}$.  
     Let     $v\in \bigcap_{i=1}^{m}{\mathfrak{q}^{s}_i} \setminus \mathfrak{p}^{r+s}_{t}$. Since  $u^{r+s}\in \bigcap_{i=1}^{t-1}{\mathfrak{p}^{r+s}_i}$, this gives that     $u^{r+s}v\in \bigcap_{i=1}^{t-1}{\mathfrak{p}^{r+s}_i} \cap \bigcap_{i=1}^{m}{\mathfrak{q}^{s}_i}$, and so $u^{r+s}v\in \mathfrak{p}^{r+s}_{t}$. Based on     \cite[Corollary 6.1.8]{V1}, it is well-known that  $\mathfrak{p}^{r+s}_{t}$ is $\mathfrak{p}_{t}$-primary. It follows now from     $u^{r+s}v\in \mathfrak{p}^{r+s}_{t}$ that $v\in \mathfrak{p}^{r+s}_{t}$  or  $u^{r+s}\in \sqrt{\mathfrak{p}^{r+s}_{t}}$.
     Since  $v\notin \mathfrak{p}^{r+s}_{t}$, we must have $u^{r+s}\in \sqrt{\mathfrak{p}^{r+s}_{t}}$. 
     On account of $\sqrt{\mathfrak{p}^{r+s}_{t}}=\mathfrak{p}_{t}$ and $\mathfrak{p}_{t}$ is a monomial prime ideal, we get 
     $u\in \mathfrak{p}_{t}$, which is a contradiction. Therefore,  $\bigcap_{i=1}^{t}{\mathfrak{p}^{r+s}_i} \cap \bigcap_{i=1}^{m}{\mathfrak{q}^{s}_i}$ is a minimal primary decomposition of  $I^{r+s} \cap J^s$, and thus $\mathrm{Ass}(I^rJ^s)=\mathrm{Ass}(I^{r+s}\cap J^s)=\mathrm{Min}(I) \cup \Gamma$ 
     for all $r,s\geq 1$. To complete the proof, put $L:=I^rJ^s$ and let $\ell \geq 1$.  Because $L^\ell=I^{r\ell} J^{s\ell}$, the  argument above implies that 
     $\mathrm{Ass}(L^\ell)=\mathrm{Min}(I) \cup \Gamma=\mathrm{Ass}(L)$. Consequently, we can deduce that $L$ is normally torsion-free, as required.  
\end{proof}


Here, we present the next result from this section, which deals with the relation between  the demotion property and normally torsion-freeness. 

\begin{theorem}  \label{NTT-2}
Suppose that  $I$ and $J$  are    square-free monomial ideals in $R=K[x_1, \ldots, x_n]$ such that    $I=J+(x_\alpha x_\beta)$, where 
 $1\leq \alpha \neq \beta \leq n$. Then the following statements hold:
\begin{itemize}
\item[(i)] if $J$ is  normally torsion-free, then  $J$ is a proper demotion of $I$. 
\item[(ii)] if both $I$ and  $J$ are   normally torsion-free,  $h:=\prod_{i=1}^mx_{n+i}$ and  $S=R[x_{n+1}, \ldots, x_{n+m}]$, then $L:=JS+ hx_\alpha x_\beta S$  is a   normally  torsion-free square-free monomial ideal in $S$.   
\end{itemize}
\end{theorem}
 
 \begin{proof}
(i)  Set  $J:=(u_1, \ldots, u_t)$.  We want to show that  $I^rJ^s=I^{r+s} \cap J^s$ for all $r,s \geq 0$. 
 Clearly,  the  claim is true for $r=0$ or    $s=0$ or both of them, so we fix  $r,s \geq 1$.   Due to  $I=J+(x_\alpha x_\beta)$, this implies  the following equalities
  \begin{align*}
  I^{r+s}\cap J^s=& (J+ (x_\alpha x_\beta))^{r+s} \cap J^s\\
  =& (\sum_{c=0}^{r+s} J^{c}(x_\alpha x_\beta)^{r+s-c}) \cap J^s\\
  =& \sum_{c=0}^{r+s} (J^{c}(x_\alpha x_\beta)^{r+s-c} \cap J^s)\\
  =& \sum_{c=0}^{s-1} (J^{c}(x_\alpha x_\beta)^{r+s-c} \cap J^s) +
   \sum_{c=s}^{r+s} (J^{c} (x_\alpha x_\beta)^{r+s-c} \cap J^s).
  \end{align*}
  Let $s\leq c \leq r+s$. Hence,  $J^{c} (x_\alpha x_\beta)^{r+s-c}  \subseteq J^s$, and we therefore get  
  $$\sum_{c=s}^{r+s} (J^{c} (x_\alpha x_\beta)^{r+s-c} \cap J^s)= \sum_{c=s}^{r+s} J^{c} (x_\alpha x_\beta)^{r+s-c} 
   =J^s \sum_{\rho=0}^{r} J^{\rho} (x_\alpha x_\beta)^{r- \rho}= J^sI^r.$$
    We claim that    $J^{c} (x_\alpha x_\beta)^{r+s-c} \cap J^s \subseteq J^s  (x_\alpha x_\beta)^{r}$ for each  $0\leq c \leq s$. 
  Let $0 \leq c \leq s$. By virtue of   $J$ is a normally torsion-free square-free monomial ideal,  
  Facts \ref{proposition 4.3.29}, \ref{exercise 6.1.25} and Theorem \ref{Villarreal1} imply that
   $J^c = J^{(c)}= \bigcap_{\mathfrak{p}\in \mathrm{Min}(J)} \mathfrak{p}^{c}$   and  $J^s = J^{(s)}= \bigcap_{\mathfrak{p}\in \mathrm{Min}(J)} \mathfrak{p}^{s}$. By  part (a) of Fact \ref{exercise 7.9.1}, we can deduce $ J^c (x_\alpha x_\beta)^{r+s-c}= \bigcap_{\mathfrak{p}\in \mathrm{Min}(J)}((x_\alpha x_\beta)^{r+s-c}  \mathfrak{p}^{c})$, and so  the following equalities
  \begin{align*}
   J^{c}(x_\alpha x_\beta)^{r+s-c} \cap J^s =&    \left((x_\alpha x_\beta)^{r+s-c} \bigcap_{\mathfrak{p}\in \mathrm{Min}(J)} \mathfrak{p}^{c}\right) \cap 
     \left(\bigcap_{\mathfrak{p}\in \mathrm{Min}(J)} \mathfrak{p}^{s}\right)\\
  =& \left(\bigcap_{\mathfrak{p}\in \mathrm{Min}(J)}((x_\alpha x_\beta)^{r+s-c}  \mathfrak{p}^{c})\right) \cap 
  \left(\bigcap_{\mathfrak{p}\in \mathrm{Min}(J)} \mathfrak{p}^{s}\right)\\
  =& \bigcap_{\mathfrak{p}\in \mathrm{Min}(J)}\left(((x_\alpha x_\beta)^{r+s-c}  \mathfrak{p}^{c}) \cap 
   \mathfrak{p}^{s}\right).
  \end{align*}
  Choose an arbitrary element   $\mathfrak{p}\in \mathrm{Min}(J)$. In what follows, our aim is to prove 
  $((x_\alpha x_\beta)^{r+s-c} \mathfrak{p}^{c}) \cap    \mathfrak{p}^{s} \subseteq (x_\alpha x_\beta)^{r}  \mathfrak{p}^{s}$. 
     Take  $v\in \mathcal{G}(((x_\alpha x_\beta)^{r+s-c}  \mathfrak{p}^{c}) \cap    \mathfrak{p}^{s})$. Then 
    $v=\mathrm{lcm}((x_\alpha x_\beta)^{r+s-c}f,g)$, 
    where $f\in \mathcal{G}({\mathfrak{p}}^{c})$ and $g\in  \mathcal{G}({\mathfrak{p}}^{s})$.  Because   $f\in \mathcal{G}({\mathfrak{p}}^{c})$ (respectively, $g\in  \mathcal{G}({\mathfrak{p}}^{s})$), one may  write $f=\prod_{i=1}^nx^{e_i}_i$ 
    (respectively,   $g=\prod_{i=1}^nx^{d_i}_i$) such that  $e_i\geq 0$ for all $i=1, \ldots, n$, $\sum_{i=1}^n e_i=c$, and $x_i\in \mathfrak{p}$ when 
    $e_i>0$     (respectively, $d_i\geq 0$ for all $i=1, \ldots, n$, $\sum_{i=1}^n d_i=s$, and $x_i\in \mathfrak{p}$ when $d_i>0$). 
        It follows  readily   from  $v=\mathrm{lcm}((x_\alpha x_\beta)^{r+s-c}f,g)$ that     
  $$v=\left(\prod_{i=1, i\notin\{\alpha, \beta\}}^{n} x_{i}^{\max\{e_i,d_i\}}\right) x_{\alpha}^{\max\{r+s-c + e_\alpha, d_\alpha\}}  
   x_{\beta}^{\max\{r+s-c + e_\beta,d_\beta\}}.$$
   Since $\deg_{x_{\alpha}} v= \max\{r+s-c + e_\alpha, d_\alpha\}$ and $\deg_{x_{\beta}} v=\max\{r+s-c + e_\beta, d_\beta\}$, we get  
      $\deg_{x_{\alpha}} v \geq r+s-c +e_\alpha \geq r$ and $\deg_{x_{\beta}} v \geq r+s-c + e_\beta \geq r$. By virtue of 
   $\sum_{i=1}^n e_i=c$ and $\deg_{x_i} v =\max\{e_i, d_i\}$ for all $i\in \{1, \ldots, n\}\setminus \{\alpha, \beta\}$, we can derive    
   $$\sum_{i=1}^n\deg_{x_i} v \geq \sum_{i=1, i\notin\{\alpha, \beta\}}^ne_i +r+s-c +e_\alpha + r+s-c+e_\beta \geq s+2r.$$
   From   $\deg_{x_{\alpha}} v \geq r$, $\deg_{x_{\beta}} v \geq  r$, and $\sum_{i=1}^n\deg_{x_i} v \geq s+2r$, we have  
   $v\in (x_\alpha x_\beta)^{r}  \mathfrak{p}^{s}$. This gives   $((x_\alpha x_\beta)^{r+s-c} \mathfrak{p}^{c}) \cap    \mathfrak{p}^{s} \subseteq (x_\alpha x_\beta)^{r}  \mathfrak{p}^{s}$  for all   $\mathfrak{p}\in \mathrm{Min}(J)$,  and so  $J^{c}(x_\alpha x_\beta)^{r+s-c} \cap J^s \subseteq J^s(x_\alpha x_\beta)^{r}$
  for all $0\leq c \leq s$. This yields that     $I^rJ^s=I^{r+s} \cap J^s$ for all $r,s \geq 0$, that is, $J$ is a proper demotion of $I$. 
 
(ii)  According to  Theorem \ref{Villarreal1}, it is enough to show that  $L^k=L^{(k)}$ for all  $k\geq 1$. Fix   $k\geq 1$.  
 Since     $L=JS+ hx_\alpha x_\beta S$ and  $\mathrm{supp}(h) \cap  \mathrm{supp}(I)=\emptyset$, it follows from Fact \ref{fact1} that $L=(J, h) \cap (J, x_\alpha x_\beta)=J+ hI$.  On account of   Proposition \ref{Intersection}, we get  $L^{(k)}=(J, h)^{(k)} \cap (J, x_\alpha x_\beta)^{(k)}$.  In light of $J$ is normally torsion-free and  $\mathrm{supp}(h) \cap  \mathrm{supp}(J)=\emptyset$, one can  deduce from Theorem \ref{NTF.Th.2.5}  that $(J, h)$ is normally 
 torsion-free, and hence $(J, h)^{(k)}=(J, h)^{k}$.  Since $I$ is normally torsion-free, we get $(J, x_\alpha x_\beta)^{(k)}= (J, x_\alpha x_\beta)^{k}$.
  This gives rise to    $L^{(k)}=(J, h)^{k} \cap (J, x_\alpha x_\beta)^{k}$.  It follows now from part (i),  Fact \ref{fact1} and  the binomial expansion theorem
   that 
 \begin{align*}
  (J, h)^k  \cap (J, x_\alpha x_\beta)^k=& (J, h)^k  \cap I^k=   \left(\sum_{\alpha=0}^{k}h^{k-\alpha}J^{\alpha}\right) \cap  I^k\\
  = &\sum_{\alpha=0}^{k} (h^{k-\alpha}J^{\alpha} \cap  I^k) \overset{(\text{i})}{=}  \sum_{\alpha=0}^{k} h^{k-\alpha}J^{\alpha}   I^{k-\alpha} \\
  =& (J+ hI)^k\\
  =& L^k.
  \end{align*}
Therefore,   we get $L^{k}=L^{(k)}$, and  Theorem \ref{Villarreal1} yields  that $L$ is normally torsion-free, as required. 
 \end{proof} 
 

\begin{example}
\em{
Let   $G$ be  a  bipartite graph with $V(G)=\{x_1, \ldots, x_n\}$ and $I=(u_1, \ldots, u_t, u_{t+1})$ be  its edge ideal. 
 Put $J:=(u_1, \ldots, u_t)$. Based on Theorem  \ref{NTT-2},   $J$ is  a demotion of $I$.
 }
 \end{example}


\section{Comparing reductions and demotions of ideals} \label{Section 7}

This section is mainly devoted to a comparison between reductions and demotions of ideals.  We start by revisiting the definition of reductions of ideals, and then establish that a number of its characteristic properties are not preserved in the case of demotions of ideals. For this purpose, we present several counterexamples.

 \begin{definition}(\cite[Definition 1.2.1]{HS})
\em{
Let   $J \subseteq I$ be ideals. Then  $J$ is said to be a \textit{reduction} of $I$ if there exists a nonnegative integer $n$ such that 
$I^{n+1} = JI^n$.
}
\end{definition}


\begin{example} \label{Example-1} 
\em{
(i) Let $R=K[x_1, \ldots, x_7]$, $J:=(x_2x_4, x_2x_5, x_1x_4, x_5x_6, x_4x_7)$ and $I:=J+(x_1x_3)$. 
Consider the graph  $G=(V(G), E(G))$ with vertex set  $V(G)=\{x_1, x_2, x_4, x_5, x_6, x_7\}$ and the following edge set 
$$E(G)=\{\{x_2,x_4\}, \{x_2,x_5\}, \{x_1,x_4\}, \{x_5,x_6\}, \{x_4,x_7\}\}.$$
 Since $J$ is the edge ideal of the bipartite graph $G$, this implies that $J$ is normally torsion-free. 
It follows now from Theorem \ref{NTT-2} that $J$ is  a demotion of $I$. In what follows, we show   $I^{n+1} \neq I^n  J$ for every $n\geq 0$.
 Consider the monomial $m=x_1^{n+1} x_3^{n+1} \in I^{n+1}.$  
Suppose, for contradiction, that \(m \in I^n  J\). Then there exist monomials \(f \in I^n\) and \(g \in J\) such that $m = f  g.$ 
Note that \(f\) is a product of \(n\) generators of \(I\), and each generator of \(I\) is either \(x_1 x_3\) or one of the generators of \(J\). Therefore, \(f\) can only contain variables from \(\{x_1, x_2, x_3, x_4, x_5, x_6, x_7\}\), but to produce \(x_1^{n+1} x_3^{n+1}\) exactly, \(f\) must be a power of \(x_1 x_3\) alone.  
On the other hand, every generator of \(J\) contains at least one variable outside \(\{x_1, x_3\}\), that is, $x_2x_4,\, x_2x_5,\, x_1x_4,\, x_5x_6,\, x_4x_7.$  
Hence, any \(g \in J\) necessarily introduces a variable not in \(\{x_1, x_3\}\). Multiplying \(f\) by such a \(g\) would produce a monomial containing an ``extra" variable, which contradicts the fact that \(m\) involves only \(x_1\) and \(x_3\).  Therefore, no such factorization \(m = f  g\) exists, and we conclude
 $m \notin I^n  J.$ Consequently, $I^{n+1} \neq I^n  J$ for all $n\geq 0$. This shows that $J$ is not a reduction of $I$. 

(ii) Suppose that   $I=(x^3, y^3, z^3, x^2y, xy^2, y^2z, yz^2, x^2z, xz^2)$ and $J=(x^3, y^3, z^3)$ 
in the polynomial ring $R=K[x,y,z]$. Using  \textit{Macaulay2} \cite{GS}, we find that  $I^3=JI^2$, while $IJ^3\neq I^4\cap J^3$. Hence, 
we can deduce that $J$ is a reduction of $I$, but $J$ is not a demotion of $I$.  
}
\end{example}


\begin{proposition}(\cite[Proposition 1.2.4]{HS}) \label{reduction-1}
Let $K \subseteq J \subseteq I$ be ideals in $R$. 
\begin{itemize}
\item[(i)] If $K$ is a reduction of $J$ and $J$ is a reduction of $I$, then $K$ is a reduction of $I$.
\item[(ii)] If $K$ is a reduction of $I$, then $J$ is a reduction of $I$.
\item[(iii)] If $I$ is finitely generated, $J = K + (r_1, \ldots, r_k)$, and K is a reduction of   $I$, then $K$ is a reduction of $J$.
\end{itemize}
\end{proposition}


\begin{example}\label{Example-2}
\em{
Let $I=(x_3,x_1x_2,x_4x_5x_6)$ be a square-free monomial ideal in $R=K[x_1, \ldots, x_8]$. According to Theorem  \ref{NTF.Th.2.5}, we can derive that $I$ is normally torsion-free. Moreover, it is easy to check that 
\begin{align*}
I=& (x_1,x_3,x_4) \cap (x_1,x_3,x_5) \cap(x_1,x_3,x_6) \cap (x_2,x_3,x_4)  \cap (x_2,x_3,x_5)\\
& \cap(x_2,x_3,x_6).
\end{align*}
Now, set $J:=I\cap (x_1,x_7)$  and $L:=J\cap (x_4,x_8)$. It follows readily from Lemma \ref{NTF1} that both $J$ and $L$ are normally torsion-free as well. 
 Using Proposition \ref{Demotion-3}, we can conclude that  $J$ is  a demotion of $I$, and  $L$ is  a demotion of $J$. Since 
 $(x_1x_2x_4x_5x_6)^2 \in I^4\cap L^2 \setminus I^2L^2$, this implies that   $L$ is not a demotion of $I$. 
  This demonstrates that statement (i) in Proposition \ref{reduction-1} does not hold for demotions of ideals in general.
}
\end{example}


\begin{example}\label{Example-3}
\em{
Let $R=K[x_1, \ldots, x_7]$ and consider the following square-free monomial ideal in $R$
\begin{align*}
I=&(x_3,x_1x_2,x_4x_5x_6)= (x_1,x_3,x_4) \cap (x_1,x_3,x_5) \cap(x_1,x_3,x_6)\\
& \cap (x_2,x_3,x_4)  \cap (x_2,x_3,x_5)  \cap(x_2,x_3,x_6).
\end{align*}
Moreover, assume  that $J:=I\cap (x_4,x_5,x_7)$  and  $L:=I\cap (x_4,x_7)$. It is routine to check that $L\subseteq J \subseteq I$.  We can deduce  from 
 Theorem  \ref{NTF.Th.2.5} that  $I$ is normally torsion-free, and by Lemma \ref{NTF1}, we obtain  $L$ is also   normally torsion-free.
  It follows immediately from Proposition \ref{Demotion-3} that 
 $L$ is  a demotion of $I$. On account of $x_3^3x_4x_5x_6 \in I^4\cap J^2 \setminus I^2J^2$, we can derive that  $J$ is not  a demotion of $I$. 
This shows that statement (ii) in Proposition \ref{reduction-1} is, in general, not valid for demotions of ideals.   
}
\end{example}


\begin{lemma}(\cite[Lemma 8.1.10]{HS})\label{reduction-2}
Let $J \subseteq I$ be a reduction. Then $\sqrt{I} = \sqrt{J},$ $\mathrm{Min}(R/I) = \mathrm{Min}(R/J),$ and $\mathrm{ht}(I) = \mathrm{ht}(J).$
\end{lemma}


\begin{example}\label{Example-4}
\em{
Let $R=K[x_1, \ldots, x_7]$, $J:=(x_2x_4, x_2x_5, x_1x_4, x_5x_6, x_4x_7)$ and $I:=J+(x_1x_3)$.  
According to Example \ref{Example-1}(i),  $J$ is  a demotion of $I$. On the other hand, using  \textit{Macaulay2} \cite{GS}, we detect that
$$\mathrm{Ass}(R/J)=\{(x_5,x_4), (x_6,x_4,x_2), (x_7,x_5,x_2,x_1), (x_7,x_6,x_2,x_1)\} \; \text{and}$$
\begin{align*}
\mathrm{Ass}(R/I)=\{&(x_5,x_4,x_1), (x_5,x_4,x_3), (x_6,x_4,x_2,x_1), (x_7,x_5,x_2,x_1),\\
& (x_7,x_6,x_2,x_1), (x_6,x_4,x_3,x_2)\}.
\end{align*}
Hence, we get  $\mathrm{ht}(J)=2$,  $\mathrm{ht}(I) = 3$, $(x_4,x_5) \in  \mathrm{Min}(J)\setminus \mathrm{Min}(I)$, and  
 $(x_1, x_4, x_5) \in \mathrm{Min}(I)\setminus \mathrm{Min}(J)$. It is also obvious that  $\sqrt{I} \neq  \sqrt{J}$. 
This indicates that Lemma  \ref{reduction-2}  may generally fail for demotions of ideals.
 }
 \end{example}
 

 \begin{proposition}(\cite[Proposition 8.1.5]{HS})\label{reduction-3}
  Let  $J=(a_1, \ldots, a_k) \subseteq I$  be ideals in a ring $R$.
  \begin{itemize}
  \item[(1)]  If  $J$ is a reduction of  $I$, then for any positive integer $m$, $(a^m_1, \dots,  a^m_k)$ and $J^m$ are reductions of $I^m$.
  \item[(2)]   If for some positive integer $m$, $(a^m_1, \ldots, a^m_k)$ or  $J^m$ is a reduction of  $I^m$, then $J$ is a reduction of  $I$.
  \end{itemize}
\end{proposition}
 
We close this paper with the following example. 
 \begin{example}\label{Example-5}
 \em{
  Let  $J=(x_2x_4, x_2x_5, x_1x_4, x_5x_6, x_4x_7)$ and $I=J+(x_1x_3)$ in  $R=K[x_1, \ldots, x_7]$. 
  It follows from Example \ref{Example-1}(i)  that  $J$ is  a demotion of $I$.  Now, put  $L:=((x_2x_4)^2, (x_2x_5)^2, (x_1x_4)^2, (x_5x_6)^2, (x_4x_7)^2)$. 
  Because  $(x_1x_3)(x_2x_4)^2 \in I^3 \cap L \setminus I^2L$, we can conclude  that $L$ is not a demotion of $I^2$. This states that it is possible 
  $J=(a_1, \ldots, a_k)$  is a demotion of  $I$, but  $(a^m_1, \dots,  a^m_k)$ is not  a demotion  of $I^m$ for some $m\geq 1$ (see statement (1) in Proposition 
  \ref{reduction-3}).
  }
\end{example}


\end{document}